\documentclass[letterpaper, twoside, 11pt]{amsart}
\usepackage[T1]{fontenc}
\usepackage{lmodern}
\usepackage[utf8]{inputenc}	
\usepackage[main=english]{babel}
\usepackage{geometry}
\usepackage{comment}

\usepackage{xcolor}
\definecolor{forestgreen}{HTML}{228B22}
\definecolor{royalblue}{HTML}{4169E1}

\usepackage[unicode, pdftex]{hyperref}
\hypersetup{
	colorlinks   = true, 
	urlcolor     = blue, 
	linkcolor    = royalblue, 
	citecolor    = forestgreen 
}

\usepackage[noabbrev,capitalise,nameinlink]{cleveref}

\usepackage[abbrev,alphabetic,backrefs,initials,nobysame]{amsrefs}

\usepackage{amssymb}
\usepackage{tikz-cd}
\usepackage[all]{xy}
\usepackage{mathrsfs}

\newcommand{\longdashleftrightarrow}[1][1.75em]{\mathrel{
		\tikz[baseline]{\draw[dash pattern=on .25em off .1 em,<->](0,.58ex)--(#1,.58ex)}}}

\newtheorem{theorem}{Theorem}[section]
\newtheorem{proposition}[theorem]{Proposition}
\newtheorem{lemma}[theorem]{Lemma}
\newtheorem{corollary}[theorem]{Corollary}
\newtheorem{conjecture}[theorem]{Conjecture}

\theoremstyle{definition}
\newtheorem{definition}[theorem]{Definition}
\newtheorem{construction}[theorem]{Construction}
\newtheorem{example}[theorem]{Example}

\theoremstyle{remark}
\newtheorem{remark}[theorem]{Remark}
\newtheorem{question}[theorem]{Question}

\theoremstyle{plain}
\newcounter{intro}

\newtheorem{intro-conjecture}[intro]{Conjecture}
\newtheorem{intro-corollary}[intro]{Corollary}
\newtheorem{intro-theorem}[intro]{Theorem}
\newtheorem*{intro-corollary*}{Corollary}

\newcommand{\Q}{\mathbb{Q}}
\newcommand{\R}{\mathbb{R}}
\newcommand{\Ko}{\mathbb{C}}
\newcommand{\Zi}{\mathbb{Z}}

\newcommand{\Cc}{\mathfrak{C}}

\newcommand{\X}{\mathcal{X}}
\newcommand{\B}{\mathcal{B}}
\newcommand{\Z}{\mathcal{Z}}
\newcommand{\Y}{\mathcal{Y}}
\newcommand{\V}{\mathcal{V}}
\newcommand{\Qp}{\mathcal{Q}}

\newcommand{\Pu}{\mathcal{P}}
\newcommand{\Li}{\mathcal{L}}

\newcommand{\Ho}{\mathcal{H}}

\newcommand{\Os}{\mathcal{O}}
\newcommand{\A}{\mathcal{A}}

\DeclareMathOperator{\Sp}{Spec}
\newcommand{\Pp}{\mathbb{P}}
\newcommand{\al}{\alpha}
\newcommand{\Aut}{\textnormal{Aut}}
\newcommand{\gr}{\textnormal{gr}\,}

\DeclareMathOperator{\supp}{supp}
\newcommand{\Hilb}{\mathbf{HIlb}_{\Pp^N}^{\mathbf{P}(t)}}
\newcommand{\Pic}{\textbf{Pic}_Z}
\newcommand{\LS}{\textnormal{LS}_Z}
\newcommand{\VA}{\textnormal{VA}}
\newcommand{\PGL}{\mathbf{PGL}_{N+1}}
\newcommand{\D}{\mathbb{D}}

\newcommand{\CY}{\textnormal{CY}}
\newcommand{\an}{\textnormal{an}}

\newcommand{\Iso}{\textnormal{Iso}}
\newcommand{\Wlc}{\textnormal{Wlc}}

\newcommand{\Amp}{\textnormal{Amp}}
\newcommand{\Nef}{\textnormal{Nef}}
\newcommand{\Mov}{\textnormal{Mov}}

\let\emptyset\varnothing

\title{Constructibility aspects of the cone conjecture}
\author[D. V. Serebrennikov]{Daniil Serebrennikov}
\address{Department of Mathematics, Johns Hopkins University, Baltimore, MD 21218, USA.}
\email{dserebr1@jhu.edu}
\date{}

\begin{document}
	\begin{abstract}
		We establish two consequences of the Kawamata--Morrison--Totaro cone conjecture, and prove them unconditionally in all dimensions. First, for a K-torsion  variety, the natural action of its automorphism  group on the set of ample divisor classes of fixed volume has only finitely many orbits. Second, the number of (isomorphism classes of) minimal models for a given K-torsion variety is finite if these models admit a bounded polarization.
	\end{abstract}
	\maketitle
	\tableofcontents
	\section{Introduction}
	In this paper, we study families of projective Calabi--Yau pairs with klt singularities over the field of complex numbers $\Ko$. For the sake of clarity, we state our results for \text{K-torsion} varieties. We say that a smooth projective variety $X$ is \textit{K-torsion} if the canonical line bundle $\omega_X$ is torsion. The class of K-torsion varieties includes classical examples such as the Fermat quartic surface in $\Pp^3_\Ko$, i.e. a K3 surface, a general quintic hypersurface in $\Pp^4_\Ko$, an abelian variety, etc.
	The main conjecture concerning the birational geometry of K-torsion varieties is the Kawamata--Morrison--Totaro cone conjecture. The latter conjecture is an unraveled problem in dimension~$\ge 3$. We provide an informal diagram that summarizes the conjecture and its consequences (see \cref{Appendix Cone_Conjecture}).\vfil
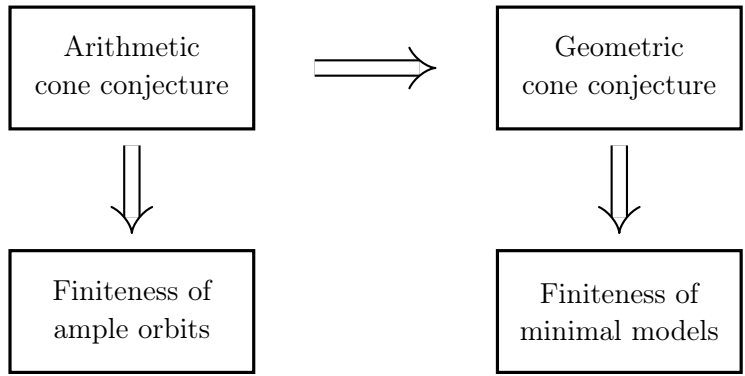
		\begin{center}
		\begin{tikzpicture}
			\def\sep{2pt} 
			\def\vhalf{3pt} 
			\draw[black, very thick] (-1,0) rectangle (2.24,1.62);
			\node[align=center] at (0.62,0.81) {Arithmetic\\ cone conjecture};
			
			\draw[black, very thick] (5.48,0) rectangle (8.72,1.62);
			\node[align=center] at (7.1,0.81) {Geometric\\ cone conjecture};
			\draw[
			black, thick,
			double, double distance=5pt,
			line cap=butt, line join=miter,
			-{Implies[length=8pt,width=10pt]}
			] (3.05,0.81) -- (4.67,0.81);
			
			\draw[black, very thick] (-1,-3.24) rectangle (2.24,-1.62);
			\node[align=center] at (0.62,-2.43) {Finiteness of\\ ample orbits};
			\draw[
			black, thick,
			double, double distance=5pt,
			line cap=butt, line join=miter,
			-{Implies[length=8pt,width=10pt]}
			] (0.62,-0.22) -- (0.62,-1.4);
			
			\draw[black, very thick] (5.48,-3.24) rectangle (8.72,-1.62);
			\node[align=center] at (7.1,-2.43) {Finiteness of\\ minimal models};
			\draw[
			black, thick,
			double, double distance=5pt,
			line cap=butt, line join=miter,
			-{Implies[length=8pt,width=10pt]}
			] (7.1,-0.22) -- (7.1,-1.4);
		\end{tikzpicture}
	\end{center}
	
	The arithmetic cone conjecture predicts (cf. \cref{prop: arithmetic}) the finiteness of ample orbits, which is our first result. To be precise, let $\text{NS}(X) = \mathrm{Pic}(X)\mathbin{/}\!\equiv$ denote the group of line bundles on $X$ modulo the numerical equivalence. Then there is a natural (right) action of the automorphism group $\text{Aut}(X)$ of $X$ on $\text{NS}(X)$ given by $[A]\cdot \varphi = [\varphi^* A]$ for every $\varphi \in \text{Aut}(X)$, and every $[A]\in \text{NS}(X)$.
	\begin{intro-theorem}[Finiteness of ample orbits]
		\label{intro-theorem: orbits}
		Fix $N\in \mathbb{Z}_{>0}$. Let $X$ be  a smooth projective K-torsion variety over $\Ko$. Then the number of orbits for the natural action of $\textnormal{Aut}(X)$ on the set $$\{[A]\in \textnormal{NS}(X): A^{\dim X} = N, \ A \text{ is ample}\}$$ is finite (whenever the defined set is non-empty).
	\end{intro-theorem}
	Besides particular cases for which the cone conjecture is proven, \cref{intro-theorem: orbits} is known for abelian varieties \cite[Theorem 1.1]{NN81}, smooth K3 surfaces \cite[Proposition 2.6]{Ste85}, smooth Calabi--Yau threefolds \cite[Corollary 4.5]{Sze99}, i.e. $\omega_X \approx \Os_X$ and $H^1(X, \Os_X) = H^2(X, \Os_X) = 0$, hyperk\"ahler manifolds \cite[Theorem 0.1]{Huy18}, and klt Calabi--Yau log surfaces \cite[Theorem 4.1]{Tot10} for which the cone conjecture is proven in full generality. In addition, it is recently shown in \cite[Lemma 2.22]{BFMT25} that a pure integral Hodge structure admits only finitely many polarizations up to isomorphism that is a generalization of \cite[Theorem 1.1]{NN81} in a slightly different direction.
	
	The geometric cone conjecture predicts (cf. \cref{prop: geometric_implies_finiteness}) that the number of smooth projective K-torsion varieties $X_i$ in a fixed birational class is finite up to isomorphism. We confirm this prediction under the assumption that these models $X_i$ \textit{admit a bounded polarization}, i.e. there exists an integer $N\in \Zi_{>0}$ such that for each variety $X_i$ there is a very ample line bundle $A_i$ on $X_i$ such that $A_i^{\dim X_i} \le N$.
	\begin{intro-theorem}[Finiteness of minimal models]
		\label{intro-theorem: finiteness}
		Let $\Cc$ be a class of birationally equivalent smooth projective K-torsion varieties over $\Ko$. Suppose that the class $\Cc$ admits a bounded polarization. Then the set of isomorphism classes in $\Cc$ is finite.
	\end{intro-theorem}
	
	It is worth mentioning that \cref{intro-theorem: finiteness} is already known for smooth Calabi--Yau threefolds \cite[Corollary 4.4]{Sze99}. However, our approach is rather a generalization of the previous work \cite[Theorem 1.4]{Ser26}, where the case of klt Calabi--Yau log surfaces was treated. Let us recall that a birational map $X \longdashleftrightarrow Y$ is called \textit{small} if the loci of indeterminacy have codimension at least two. By $S(\Ko)$ we denote the set of all closed points in $S$ with the induced Zariski topology. The main technical tool to prove \cref{intro-theorem: orbits} and \cref{intro-theorem: finiteness} is the following result.
	
	\begin{intro-theorem}[Constructibility]
		\label{intro-theorem: const}
		Let $X$ be a smooth projective K-torsion variety over $\Ko$, $f\colon \X\to S$ a smooth projective morphism between varieties over $\Ko$. Then the subsets
		\begin{equation*}
			\begin{gathered}
				\Iso_X(\X/S) = \{s\in S(\Ko): \X_s \approx X\},\\
				\Iso_X^1(\X/S) = \{s\in S(\Ko): \X_s \overset{\textnormal{small}}{\longdashleftrightarrow} X\}
			\end{gathered}
		\end{equation*}
		 are (Zariski) constructible in $S(\Ko)$.
	\end{intro-theorem}
	
	In fact, we prove analogues of \cref{intro-theorem: orbits}, \cref{intro-theorem: finiteness}, \cref{intro-theorem: const} for projective klt Calabi--Yau pairs $(X, B)$, that is, $K_X + B\equiv 0$, and $B$ is an effective $\Q$-divisor (cf. \cref{th: finiteness of orbits}, \cref{th: main}, \cref{th: Fano_0_General}). Furthermore, we prove a natural analogue of \cref{intro-theorem: orbits} for the group of birational automorphisms $\text{Bir}(X)$, that is, the number of orbits for the natural action of $\text{Bir}(X)$ on the set of big and very movable line bundles of fixed volume is finite (cf. \cref{th: finiteness of orbits(movable)}).
	
	Although \cref{intro-theorem: const} may seem natural, we show that this type of constructibily does not hold in general (see \cref{sec: counterexamples}), even for lc Calabi--Yau log surfaces. Let us remark that constructibility at least of the set $\Iso_X(\X/S)$ appeared already in the literature: hyperk\"ahler manifolds \cite[Corollary 1.7]{Huy18}, some terminal Calabi--Yau pairs of dimension 3 \cite[Theorem 4.12]{Xu25}. Eventually, \cref{th: Fano_0_General} generalizes and proves \cite[Conjecture 1.1]{Flo23}. Along the way, we establish a criterion for the constructibility. It might be of independent interest and useful in other contexts.
	
	\begin{intro-theorem}[Criterion of constructibility]
		\label{intro-theorem: constructibility_criterion}
		Let $k$ be an algebraically closed field of an arbitrary characteristic. Let  $X$ be a projective integral scheme over $k$,  $f\colon \X\to S$ a projective morphism between Noetherian separated schemes over $k$. Then the following statements are equivalent:
		\begin{itemize}
			\item
			The subset $\Iso_X(\X/S) = \{s\in S(k): \X_s \approx X\}\subseteq S(k)$ is constructible.
			\item
			There exists an ample line bundle $\A$ on ${\X}/S$, an integer $r\in \mathbb{\Zi}_{>0}$, and ample line bundles $A_1, A_2, \dots, A_r$ on ${X}$ such that for every point $s\in \Iso_{{X}}({\X}/S)$ there exists an isomorphism $\varphi_s: {X} \to {\X}_s$ with the property: $\varphi_s^*\A_s \equiv A_{i_s}$ for some $i_s\in \{1, 2, \dots r\}$.
		\end{itemize} 
	\end{intro-theorem}
	In summary, let us state a direct corollary of \cref{intro-theorem: constructibility_criterion} that gives a new evidence for the Kawamata--Morrison--Totaro cone conjecture (see \cref{cor: cone_and_constructibility} for the details).
	
	\begin{intro-corollary*}
		Keep the notation and assumptions of \cref{intro-theorem: constructibility_criterion}. The subset $\Iso_X(\X/S)\subseteq S(k)$ is constructible if $\textnormal{Nef}^e(X)$ admits a rational polyhedral fundamental domain for $\textnormal{Aut}(X)$.
	\end{intro-corollary*}
	
	\addtocontents{toc}{\protect\setcounter{tocdepth}{1}}
	\subsection*{Acknowledgments}
	The author is grateful to his advisor, Professor Shokurov, for his support and encouragement as well for several helpful discussions. The author also wishes to thank Philip Engel for explaining the geometry of anticanonical pairs, and the thoughtful comments that helped correct some inaccuracies. The author thanks Stefano Filipazzi for pointing out relevant earlier work. This paper benefited from Nazim Khelifa’s very kind suggestion on how to fill a gap in an earlier draft.
	
	\section{Preliminaries}
	\subsection{Notation and conventions.}
	Throughout, we work in the category of Noetherian schemes over $\mathbb{C}$, except \cref{sec: flag_schemes}, and follow the standard terminology \cite{IS05} unless stated otherwise. By a variety we mean an integral separated scheme of finite type over a field. On a variety, a prime divisor is understood in the Weil sense, i.e. an integral closed subscheme of codimension~$1$. In every part of the paper, for an $\mathbb{R}$-divisor $D$ the sum $D=\sum_i d_i D_i$ assumes that all $D_i$ are distinct prime divisors, and $d_i\in\mathbb{R}$.
	
	A \textit{pair} $(X,B)$ consists of a variety $X$ and an $\mathbb{R}$-divisor $B$ on $X$. We say that $(X,B)$ is a \textit{log pair} if $X$ is normal, and $K_X+B$ is an $\mathbb{R}$-Cartier $\R$-divisor. Suppose $B=\sum_i b_i B_i$, and $\Gamma\subseteq\mathbb{R}$ is a set. If $b_i\in\Gamma$ for all~$i$ then we write $B\in\Gamma$. Denote by $[0,1]$ the unit segment of real numbers.  The $\mathbb{R}$-divisor $B$ is called a \textit{boundary} if $B\in [0,1]$.	The pairs $(X,B)$ and $(X',B')$ are said to be \textit{log isomorphic} if there exists an isomorphism $\varphi:X\to X'$ such that $\varphi_*(B)=B'$; we then write $(X,B)\approx (X',B')$. We use the notation $\cong$ instead of $\approx$ if an isomorphism is natural. We write $g: (Y, D) \to (X, B)$ if $g: Y\to X$ is a morphism such that $K_Y + D = g^*(K_X + B)$.
	
	\begin{definition}
		\label{def: sing}
		Let $(X,B)$ be a log pair with $B=\sum b_i B_i$. We say that $(X,B)$ has (at worst) the following singularities when the corresponding inequalities hold:
		\[
		\begin{gathered}
			\begin{cases}
				\text{terminal (trm)} \\
				\text{Kawamata log terminal (klt)}\\
				\text{log canonical (lc)}
			\end{cases}
			\Longleftrightarrow
			\text{dis}(X,B)
			\begin{cases}
				>0;\\
				>-1 \ \text{and}\  b_i<1 \ \forall i;\\
				\ge -1 \ \text{and}\  b_i\le 1 \ \forall i.
			\end{cases}
		\end{gathered}
		\]
		Here, the \textit{discrepancy} of $(X, B)$ is given by $$\text{dis}(X, B) = \underset{E}{\inf}\{a(E; X, B): E \text{ is an exceptional prime divisor over } X\},$$ that is, $E$ runs through all the prime exceptional divisors of all birational morphisms $g: Y\to X$, and the number $a(E; X, B)\in \mathbb{R}$ is defined by the formula $K_Y = g^*(K_X + B) + \sum_E a(E; X, B)\, E$.
		Similarly, $(X, B)$ is $\epsilon$\textit{-klt} (for some $\epsilon > 0$) if $\text{dis}(X, B) > -1 + \epsilon$ and $b_i < 1$ for all $i$.
	\end{definition}
	
	\begin{definition}
		\label{def: 0-pair}
		A log pair $(X,B)$ is called a \textit{weakly log canonical model (wlc model)} if the following assumptions hold:
		\begin{itemize}
			\item $X$ is a proper variety, and $B$ is a boundary.
			\item $(X,B)$ has log canonical singularities.
			\item $K_X+B$ is nef.
		\end{itemize}
		If $(X,B)$ has (at worst) klt (resp. terminal) singularities, then $(X,B)$ is called a \textit{wlc klt model} (resp. a \textit{wlc trm model}). If, in addition, $K_X+B\equiv 0$, then $(X,B)$ is called a \textit{0-pair}\footnote{In the literature, 0-pairs are also called Calabi–-Yau pairs.}.
	\end{definition}
	
	\begin{definition}
		\label{def: 0-class}
		We say that a log pair $(X_\alpha,B_\alpha)$ is \textit{crepant birationally equivalent} or \textit{0-equivalent} to $(X,B)$ if the varieties  $X_\alpha$ and $X$ are birationally equivalent, and for each common log resolution $Y$ of these pairs there is an $\R$-divisor $D$ on $Y$ such that the following equalities hold:
		\[
		\begin{gathered}
			K_Y+D=g^*(K_X+B)=g_\alpha^*(K_{X_\alpha}+B_\alpha),\\
			B=g_*D,\quad B_\alpha=(g_\alpha)_*D,
		\end{gathered}
		\]
		where $g:Y\to X$ and $g_\alpha:Y\to X_\alpha$ are the corresponding log resolutions. The class of all log pairs crepant birationally equivalent to $(X,B)$ is called the \textit{0-class} of the log pair $(X,B)$.
	\end{definition}
	
	\begin{remark}
		If $(X,B)$ is a wlc klt model (resp. 0-pair), then each log pair $(X_\alpha,B_\alpha)$ crepant birationally equivalent to $(X, B)$ with effective $\mathbb{R}$-divisor $B_\alpha$ is also a wlc klt model (resp. 0-pair).
	\end{remark}
	
	\begin{remark}
		\label{prop: flops_0-equivalent}
		Let $(X, B), (X^\prime, B^\prime)$ be projective klt wlc models, $t: (X_\text{trm}, B_\text{trm}) \to (X, B)$ and $t^\prime: (X_\text{trm}^\prime, B_\text{trm}^\prime)\to (X^\prime, B^\prime)$ be their $\Q$-factorial terminalizations respectively \cite[Corollary 1.4.3]{BCHM10}. Suppose there is a birational map $\varphi: X \dashrightarrow X^\prime$ such that $\varphi_*B_\text{trm} = B^\prime_\text{trm}$. Then $(X_\text{trm}, B_\text{trm})$ and $(X^\prime_\text{trm}, B^\prime_\text{trm})$ are 0-equivalent according to \cite[Theorem 1]{Kaw08}. Hence $(X, B), (X^\prime, B^\prime)$ are 0-equivalent as well.
 	\end{remark}
	
	Let $f:\X\to S$ be a morphism of schemes and $s\in S$ a point. We denote the scheme-theoretic fiber by $\X_s=\X\times_S s$. If $\A$ is an invertible sheaf on $\X$, then its restriction to the fiber $\X_s$ is denoted by $\A_s=\A|_{\X_s}$. When no ambiguity arises, we denote a morphism $f: \X \to S$ simply by $\X/S$. If $\A$ is (very) ample over $S$, we  say that $\A$ is an (very) ample invertible sheaf on $\X/S$.
	
	\begin{definition}
		Suppose $f: \X \to S$ is a proper flat morphism between varieties, and the base $S$ is regular. We say that $\X/S$ is a \textit{family of varieties} if for every fiber $\X_s$ is a variety. If, in addition, the morphism $f$ is projective (resp. smooth), then $\X/S$ is called a projective (resp. smooth) family of varieties.
	\end{definition}
	
	\subsection{Bounded families}
	Consider a family of varieties $\X/S$ and an $\mathbb{R}$-divisor $\B=\sum b_i\B_i$ on $\X$. In general, the restriction of $\B$ to a fiber $\X_s$ is not defined, and even when it is, the definition is not straightforward. For our purposes it is convenient to work with \textit{elementary families} on which $\B_s$ is defined.
	
	\begin{definition}
		\label{def: elementary_family}
		Let $(\X,\B)$ be a pair and $f:\X\to S$ be a family of varieties. Suppose $\B=\sum_{i} b_i\B_i$ is a decomposition of the $\mathbb{R}$-divisor into distinct prime divisors. We say that $(\X/S,\B)$ is an \textit{elementary family} if for every point $s\in S$ the following hold:
		\begin{enumerate}
			\setcounter{enumi}{-1}
			\item The fiber $\X_s$ is a normal variety.
			\item The restriction $f|_{\B_i}:\B_i\to S$ is flat for all $i$.
			\item The closed subscheme $(\B_i)_s=\B_i\times_S s$ is a prime divisor on $\X_s$ for all $i$.
			\item The prime divisors $(\B_i)_s$ and $(\B_j)_s$ are distinct whenever $i\ne j$.
		\end{enumerate}
		In this situation, we put $\B_s=\sum_i b_i(\B_i)_s$. If, in addition, the morphism $f$ is projective, then $(\X/S,\B)$ is called a \textit{projective elementary family}.
	\end{definition}
	\begin{remark}
		\label{reamrk: restriction_canonical}
		For a given elementary family $(\X/S, \B)$ such that $(\X, \B)$ is a log pair we may assume that $(K_\X + \B)|_{\X_s} = K_{\X_s} + \B_s$ for general points $s\in S$. Indeed, by the generic smoothness on the source there is an open dense subset $\iota: U\subseteq \X$ such that $f|_{U}: U \to S$ is smooth. Moreover, $\text{codim}(\X\smallsetminus U)\ge 2$ because $\X$ is normal. Choose a general canonical divisor $K_U$. Then $K_{U_s}:= K|_{U_s}$ is a canonical divisor on $U_s$ for general points $s\in S$. Put $K_\X := \iota_* K_U$, and $K_{\X_s} := (\iota_s)_*K_{U_s}$.
	\end{remark}

	\begin{definition}
		A \textit{polarized pair} $(X, B; A)$ is a triple such that $(X, B)$ is a pair, and $A$ is an ample invertible sheaf on $X$. We say that two polarized pairs $(X, B; A)$ and $(X^\prime, B^\prime; A^\prime)$ are isomorphic if there exists an log isomorphism $\varphi: (X, B)\to (X^\prime, B^\prime)$ such that $A \approx \varphi^*A^\prime$, the latter is an isomorphism of invertible sheaves on $X$. In this situation, we write $(X, B; A) \approx (X^\prime, B^\prime; A^\prime)$.
	\end{definition}

	\begin{definition}
		\label{def: bounded_pairs}
		A class $\Cc$ of polarized pairs $(X_\alpha,B_\alpha; A_\al)$ is called \textit{bounded} if there exist finitely many elementary families $(\X^{(j)}/S^{(j)},\B^{(j)})$ together with an ample invertible sheaf $\A^{(j)}$ on $\X^{(j)}/S^{(j)}$ such that:
		\begin{itemize}
			\item For every $(X_\alpha,B_\alpha; A_\al)$ in $\Cc$ there is an index $j$ and a closed point $s_\alpha\in S^{(j)}$ such that $(X_\alpha,B_\alpha; A_\al)\approx~(\X^{(j)}_{s_\alpha},\B^{(j)}_{s_\alpha}; \A^{(j)}_{s_\alpha}).$
			\item For each index $j$ and every closed point $s\in S^{(j)}$ there exists a polarized pair $(X_\alpha,B_\alpha; A_\al)$ in $\Cc$ such that $(\X^{(j)}_s,\B^{(j)}_s, \A^{(j)}_{s})\approx~(X_\alpha,B_\alpha; A_\al)$.
		\end{itemize}
	\end{definition}

	\begin{definition}
		\label{def: bounded_polarization}
		A class $\Cc$ of projective pairs $(X_\alpha,B_\alpha)$ of fixed dimension $d\in \mathbb{Z}_{>0}$ \textit{admits a bounded polarization} if there exists a positive integer $N\in \mathbb{Z}_{> 0}$ and a finite subset $\Gamma\subset\mathbb{R}$ such that $B_\alpha\in\Gamma$, and each $X_\alpha$ admits a very ample invertible sheaf $A_\al$ on $X_\al$ satisfying the following inequalities:
		\[
		(B_\alpha)_{\mathrm{red}}\cdot A_\al^{d-1}\le N,\qquad
		A_\al^{d}\le N,
		\]
		where $(B_\alpha)_{\mathrm{red}}=\sum_{V\in\supp(B_\alpha)} V$, $\ \supp(B_\alpha)=\{\,B_{i\alpha}\mid b_{i\alpha}\ne 0\,\}$, and $B_\alpha=\sum_i b_{i\alpha}B_{i\alpha}$. The class of triples$(X_\al, B_\al; A_\al)$ is called \textit{a class of polarized pairs associated with $\Cc$ (and $N$)}.
	\end{definition}
	
	\begin{remark}
		\label{remark: components}
		As the intersection $(B_\alpha)_{\mathrm{red}}\cdot A_\al^{\dim X_\alpha-1}$ is bounded, the set $\{\#\supp(B_\alpha)\}\subseteq \Zi_{\ge 0}$ is also bounded. We write $\# \square$ for the number of elements of a finite set $\square$.
	\end{remark}
	
	 For technical reasons, flag schemes behave better in families than pairs. In the following construction we assign to every pair $(X, B)$ a flag scheme $X = Z_0\supset Z_1 \supset \dots \supset Z_n$ (cf. \cref{sec: flag_schemes}).
	\begin{construction}
		\label{const: associated_flag}
		Let $(X, B)$ be a pair. Write $B=\sum_i b_i B_i$ as a sum of distinct prime components, and set $\Gamma=\{b_i: b_i\ne 0\}$ with $n=\#\Gamma$. For each index $j=1,\dots,n$ define the reduced divisor $\overline{B}_j=\sum_{i:\, b_i=a_j} B_i$, where $a_j\in\Gamma$ and $a_j>a_{j+1}$. The finite sequence $$(a_1\overline{B}_1, a_2 \overline{B}_2, \dots, a_r\overline{B}_n)$$ is called the \textit{marked divisor associated with the $\R$-divisor} $B$.
		
		Next, for each index $j\in \{1, 2, \dots, n\}$ define a closed subset $ Z_j^{(X, B)} = \bigcup_{i= j}^n \overline{B}_i$ of $Z_0^{(X, B)} = X$ with the structure of a reduced closed subscheme. The finite sequence $$\underline{Z}^{(X, B)} = (Z_0^{(X, B)}; Z_1^{(X, B)}, \dots, Z_n^{(X, B)})$$ is called the \textit{flag scheme associated with the pair} $(X, B)$. The number $l(\underline{Z}^{(X, B)}) = n \in \mathbb{Z}_{\ge 0}$ is called \textit{the length} of $\underline{Z}^{(X, B)}$.
	\end{construction}
	\begin{remark}
		\label{remark: iso_pair_and_flag}
		The assignment $(X, B) \mapsto \underline{Z}^{(X, B)}$ is canonical. Furthermore, given two pairs $(X, B)$ and $ (X^\prime, B^\prime)$ with an isomorphism $\varphi: X\to X^\prime$ we have that $\varphi$ is a log isomorphism of pairs $(X, B), \, (X^\prime, B^\prime)$ if and only if $\varphi$ is an isomorphism of the flag schemes $\underline{Z}^{(X, B)}, \, \underline{Z}^{(X^\prime, B^\prime)}$, that is, $l\left(\underline{Z}^{(X, B)}\right) = l\left(\underline{Z}^{(X^\prime, B^\prime)}\right)$ and $\varphi\left({Z}^{(X, B)}_i\right) = {Z}^{(X^\prime, B^\prime)}_i$ for each $i = 0, 1, \dots, l\left(\underline{Z}^{(X, B)}\right)$.
	\end{remark}

		Let us recall some geometric operations frequently used in our proofs. Suppose $f: \X\to S$ is a morphism between schemes. We say that a property $P$ holds for every $s\in S$ after \textit{shrinking the base} (resp. \textit{stratification of the base}) if there exists an open subscheme $U\hookrightarrow S$ (resp. finitely many locally closed subschemes $S_i\hookrightarrow S$) such that after the base change $U\hookrightarrow S$ (resp. $\bigsqcup_i S_i \hookrightarrow S$) the property $P$ holds for every point $s\in U$ (resp. $s\in \bigsqcup_i S_i$).
	\begin{proposition}
		\label{prop: bounded_subclass}
		Let $\Cc$ be a class of pairs of fixed dimension $d \in \Zi_{>0}$. Suppose $\Cc$ admits a bounded polarization. Then any class of polarized pairs associated with $\Cc$ is a subclass of a bounded class of polarized pairs.
	\end{proposition}
	\begin{proof}
		 By the definition of a class admitting a bounded polarization there exists a finite set $\Gamma=\{0, a_1, a_2, \dots, a_n\}\subset \mathbb{R}$ such that for every $(X_\al, B_\al)$ in $\Cc$ we have $B_\alpha\in\Gamma$. Furthermore, the number of prime components of the divisors $B_\alpha$ is bounded by some $l\in\mathbb{Z}_{\ge 0}$ (see \cref{remark: components}). Let $B_\al = \sum_i b_{i\alpha}(B_\al)_i$ Then $\Cc$ splits into finitely many subclasses $\Cc_{l_1, l_2, \dots, l_n}$ corresponding to the pairs $(X_\alpha,B_\alpha)$ such that
		\[
		l_j=\#\{\,(B_{\alpha})_i\in \supp(B_\alpha):\ b_{i\alpha}=a_j\,\},\qquad 0\le\sum_{j=1}^n l_j\le l.
		\]
		Hence it suffices to prove the proposition for one fixed class $\Cc_{l_1, l_2, \dots,l_n}$. Whence we can assume $\Cc= \Cc_{l_1, l_2, \dots, l_n}, \ \sum^n_{j=1}l_j = l$, and $a_j > a_{j+1}$ for all $j$. Put $l_0 = 0$. Let $\Cc^p$ denote a class of polarized pairs associated with $\Cc$.
		
		Next, we introduce an ordered flag scheme $\underline{\widetilde{Z}}^{(X, B)}$ associated (non-canonically) to a pair $(X, B)$. Suppose $(a_1\overline{B}_1, a_2 \overline{B}_2, \dots, a_r \overline{B}_r)$ is the marked divisor associated with $B$. In addition, suppose the number of prime components of $\overline{B}_j$ equals to $l_j$. Set $m_j = \sum_{i=0}^j l_j$, and note $m_n = l$. Let $\overline{B}_j = \sum_{i= m_{j-1} + 1}^{m_j} {B}_i$ be a decomposition into prime components for every $j\in \{1, 2, \dots, n\}$. For each $j = 1, 2, \dots, l$ define a closed set $({{\widetilde{Z}}}^{(X, B)})_{j} = \bigcup_{i=j}^{l} B_i$ with the structure of a reduced closed subscheme of $({{\widetilde{Z}}}^{(X, B)})_0 = X$. In this notation, the $n$-tuple $\underline{\widetilde{Z}}^{(X, B)} = \left(({{\widetilde{Z}}}^{(X, B)})_0; ({{\widetilde{Z}}}^{(X, B)})_1, \dots, ({{\widetilde{Z}}}^{(X, B)})_l\right)$ shall be called an ordered flag scheme non-canonically attached to $(X, B)$. Nevertheless, $({\widetilde{Z}}^{(X, B)})_{m_{j-1}+1} = {Z}^{(X, B)}_{j}$ for every $j = 1, 2, \dots, n$, where $\underline{Z}^{(X, B)}$ is the flag scheme associated with $(X, B)$.
		
		By the construction above, for every $(X_\al, B_\al)$ in $\Cc$ we assign an ordered flag scheme $\underline{\widetilde{Z}_\al} = \underline{\widetilde{Z}}^{(X_\al, B_\al)}$. By the definition of a class admitting a bounded polarization, there exists a positive integer $N$ and very ample invertible sheaf $A_\alpha$ on $(\widetilde{Z_\al})_0 = X_\al$ such that $A_\al^{d}\le N$. From Matsusaka’s inequalities \cite[Chapter~VI, Exercise~2.15.8.5]{Kol96} it follows that
		\[
		\dim_\mathbb{C} H^0(X_\al,A_\al)\le A_\al^{d}+d\le N+d<+\infty.
		\]
		Hence for every polarized pair $(X_\al, B_\al; A_\al)$ there is a closed embedding $$(\widetilde{\Z}_\al)_l \subset (\widetilde{\Z}_\al)_{l-1} \subset \dots \subset (\widetilde{\Z}_\al)_{0} = X_\al \hookrightarrow \Pp^{\dim |A_\al|} \subseteq  \Pp^{N + d}_\Ko,$$ given by the complete linear system $|A_\al|$. By the standard Hilbert-Chow schemes argument (e.g. \cite[Lemma 2.20]{Bir19}), there exist finitely many $(l+1)$-tuples of polynomials $\mathbf{P}_j(t)\in \Q[t]^{\times l}$ such that for every flag scheme $\underline{\widetilde{Z}_\al}$ we have  $$\chi\left(\underline{\widetilde{Z}_\al}, A^{\otimes t}_\al\right):= \left(\chi\left(({{\widetilde{Z}_\al}})_0, A^{\otimes t}_\al\right), \chi\left(({{\widetilde{Z}_\al}})_1, A^{\otimes t}_\al\right), \dots, \chi\left(({{\widetilde{Z}_\al}})_l, A^{\otimes t}_\al\right)\right) = \mathbf{P}_{j_\al},$$ for some $j_\al$, here $\chi\left(({{\widetilde{Z}_\al}})_i, A^{\otimes t}_\al\right)$ denotes the Euler characteristic of $A_\al^{\otimes t}|_{({{\widetilde{Z}_\al}})_{i}}$ on $({{\widetilde{Z}_\al}})_{i}\subseteq X_\al$. Without loss of generality, we suppose that there is a single $(l+1)$-tuple polynomial $\mathbf{P}(t)$. By \cite[Theorem 4.5.1]{Ser06}, the flag Hilbert functor is represented by a projective scheme $\Hilb$, called the flag Hilbert scheme of $\Pp^N_\Ko$ relative to $\mathbf{P}(t)$, and by a universal family:
		$$
			\xymatrix{
			\widetilde{\Z}_l \subset \widetilde{\Z}_{l-1} \dots \subset \widetilde{\Z}_0  \subset  \Pp^N_\Ko \times \Hilb \ar@{->}[d]\\
			\Hilb.
			}
		$$
		
		Let $f_j: (\widetilde{\Z}_\al)_j \to \Hilb$ denote the restriction of the natural projection $\Pp^N_\Ko \times \Hilb \to \Hilb$. By universality, the morphism $f_j$ is flat and projective for each $j = 0, 1, \dots, l$.
		Let $S$ be the closure (with the structure of a reduced closed subscheme) of the closed points $s_\al$ corresponding to all $\underline{\widetilde{Z_\al}}$. To simplify the notation, after shrinking the base $S$ or stratification of $S$ we denote a new family of flag schemes by the same letters.
		
		According to \cite{EGAIV} (see \cite[Appendix E]{GW20}), the set $$\{s\in S: (\widetilde{\Z}_0)_s \textnormal{ is geometrically integral, and geometrically normal over $\kappa(s)$}\}$$ is open, the sets $$\{s\in S: \{\dim V: V \text{ irreducible component of $(\widetilde{Z}_j)_s$}\} = \{d-1\}\}$$ are constructible for all $j=1, \dots l$. Consequently, we may assume the above properties hold for every point $s\in S$ after stratification of the base. Since the singular locus of $S$ is a proper closed subset, we can assume that $S$ is regular. Moreover, we can assume that every connected component of $S$ is irreducible. It is not hard to prove, by decreasing induction on $j$, that the for each $j=1, \dots, l$ the scheme $\widetilde{\Z}_j$ has exactly $l-j+1$ irreducible components $\B_i$. Hence, we can write $$\widetilde{\Z}_{m_{j-1}+1} =\bigcup_{i = m_{j-1}+1}^l \B_i = \left(\bigcup_{i={m_{j-1}} + 1}^{m_j}\B_i\right)\cup \left(\bigcup_{i={m_j} + 1}^{m_{j+1}}\B_i\right)\cup \dots \cup \left(\bigcup_{i=m_{n-1}+1}^{m_n}\B_i\right).$$ Finally, let us define the $\R$-divisor $\B = \sum_{j=1}^na_j\sum_{i=m_{j-1}+1}^{m_j}\B_i$ on $\X = (\widetilde{\Z})_0$. Then for every irreducible component $S^\prime$ of $S$ the restriction $f_0|_{S^\prime}: (\X_{S^\prime}, \B_{S^\prime}) \to S^\prime$ is a projective elementary family by construction. This concludes the proof.
		
	\end{proof}
	
	\begin{remark}
		In general, a subclass of a bounded class of projective varieties need not to be bounded in the sense of \cref{def: bounded_pairs} as the following example shows. 
	\end{remark}
	\begin{example}
		Let $A_g$ be the course moduli space of all principally polarized abelian varieties of dimension $g\in \Zi_{>0}$ (see \cite[Theorem 2.1.11]{Ols11}). Consider any (infinite) countable subset $\{[X_i]\} \subset A_g$ of closed points. By \cite[p. 150]{Mum70}, for every $X_i$ there is an ample invertible sheaf $A_i$ on $X$ such that $A^g_i = g!$. By \cite[Theorem 1.1]{BS96}, $A^{\otimes 3}_i$ is very ample for all $i$.  Then $\{(X_i, A_i^{\otimes 3})\}$ is a subclass of a bounded class of projective varieties in accordance with \cref{prop: bounded_subclass}. But, the set $\{(X_i, A_i^{\otimes 3})\}$ is not bounded.
	\end{example}

	Let $(\X/S, \B)$ be an elementary family. For a given real number $\epsilon > 0$ define a set $$\textnormal{klt}_\epsilon(\X/S, \B) = \{s\in S: (\X_s, \B_s) \text{ is an $\epsilon$-klt log pair}\}.$$ In addition, if $(\X, \B)$ is a log pair then we can define the following set (cf. \cref{reamrk: restriction_canonical}) $$\nu_0(\X/S, \B) = \{s\in S: (K_\X + \B)|_{\X_s}\equiv 0\}.$$
	\begin{lemma}
		\label{lemma: klt_0_families}
		Let $(\X, \B)$ be a projective elementary family, and $\B$ is an effective $\Q$-divisor. Then for any rational number $0< \epsilon< 1$ the following hold:
		\begin{enumerate}
			\item 
				If the set ${\textnormal{klt}_\epsilon(\X/S, \B)}$ is dense in $S$ then there exists an open dense subset $U\subseteq S$ such that $$(\X_U, \B_U) = (\X, \B)\times_S U\text{ is a klt log pair}.$$
			\item 
				The subset $\textnormal{klt}_\epsilon(\X/S, \B)\subseteq S$ is constructible.
				
			\item 
				If $(\X, \B)$ is a klt log pair then the subset $\nu_0(\X/S, \B)\subseteq S$ is open.
		\end{enumerate}
	\end{lemma}
	\begin{proof}
		The statement (1) is the content of \cite[Proposition 2.4]{HX15}. Now, we show that (1) implies (2). Indeed, we can assume that $S = \overline{\textnormal{klt}_\epsilon(\X/S, \B)}$. By Noetherian induction on closed subsets of $S$, it suffices to show that $\textnormal{klt}_\epsilon(\X/S, \B)$ includes an open dense subset of $S$. By (1), we can assume that $(\X, \B)$ is a klt log pair. Let $g: (\Y, \B_\Y) \to (\X, \B)$ be a log resolution, $K_\Y + \B_\Y = g^*(K_\X + \B)$. Shrinking the base, we may suppose that $(\Y/S, \B_\Y)$ is a projective elementary family, and $(\Y, \B_\Y)$ is log smooth over $S$. Then $g_s: (\Y_s, (\B_\Y)_s) \to (\X_s, \B_s)$ is a log resolution for all $s\in S$. By construction, $g^*(K_\X + \B)|_{\Y_s}=(K_\Y + \B_\Y)|_{\Y_s} = (K_{\Y_s} +( \B_\Y)_s) = g_s^*(K_{\X_s} + \B_s)$. This implies that $\text{dis}(\X_s, \B_s) = \text{dis}(\X, \B)$ for all $s\in S$, i.e. $(\X_s, \B_s)$ is a klt log pair for all $s\in S$.
		
		Next, we prove (3). By (2), we may assume that $\textnormal{klt}_\epsilon(\X/S, \B) = S$ (for some $\epsilon>0$). To be definite, we suppose $\nu_0(\X/S, \B) \neq \emptyset$. According to \cite[Proposition 1.4.14]{Laz04} the sets $U_+ = \{s\in S: (K_\X + \B)|_{\X_s} \text{ is nef}\}$ and $U_-\{s\in S: -(K_\X + \B)|_{\X_s} \text{ is nef}\}$ are (at most) countable intersection of open subsets of $S$. In particular, $\nu_0(\X/S, \B) = U_+\cap U_-$ is dense in $S$, and contains the generic point $\eta\in S$.  By Lefschetz's principle, and abundance \cite[Theorem 4.2]{Amb05} for the geometric generic fiber $\X_{\overline{\eta}}$ we have $K_{\X_{\overline{\eta}}} + \B_{\overline{\eta}} \sim_\Q 0$. Then there is a rational function $\varphi$ on $\X_\eta$ such that $m(K_{\X_\eta} + \B_\eta) = \text{div}_{\X_\eta}(\varphi)$ for $m\gg 1$. Clearly, $\varphi$ can be viewed as a rational function on $\X$. Hence, $E = m(K_\X + \B) - \text{div}_{\X}(\varphi)$ is a vertical (over $S$) divisor. Since $(\X/S, \B)$ is a projective elementary family, all its fibers are irreducible and reduced. Hence, there is a $\Q$-Cartier $\Q$-divisor $L$ on $S$ such that $E = f^*(L)$. This implies that $K_\X + \B \sim_{S, \Q} 0$, and $\nu_0(\X/S, \B) = S$. 
	\end{proof}	

	\section{Period maps}
	\subsection{Birational Hodge Theory}
	We start off this subsection by proving a particular case of \cref{intro-theorem: const} to illustrate main ideas (cf. \cref{th: particular_case}). And rest of the section is merely a technical preparation for the generalization of these ideas.
	\begin{definition}
		Let $S$ be a locally Noetherian topological space. A subset $S^\prime\subseteq S$ is called \textit{constructible} if there exist finitely many open subsets $U_i\subseteq S$ and closed subsets $Z_i\subseteq S$ such that $S^\prime=\bigcup_i (U_i\cap Z_i)$.
	\end{definition}
	Let $S$ be a scheme. By $S^\an$ (resp. $S(\Ko)$) denote the underlying set of closed points endowed with analytic topology (resp. induced Zariski topology). Suppose $\X/S$ is a family of varieties, $X$ is a variety. Define the following sets:
	\begin{equation*}
		\Iso_X(\X/S) = \{s\in S(\Ko): \X_s \approx X\} , \ \Iso_X^1(\X/S) = \{s\in S(\Ko): \X_s \overset{\textnormal{small}}{\longdashleftrightarrow} X\}.
	\end{equation*}
	It is clear that $\Iso_X(\X/S) \subseteq \Iso_X^1(\X/S)\subseteq S(\Ko)$.
	We raise the following question:
	\begin{question}
		Let $\X/S$ be a projective family of varieties, and $X$ be a projective variety. Are the sets $\Iso_X(\X/S), \Iso_X^1(\X/S)$  constructible (Zariski/analytically)?
	\end{question}
	\begin{theorem}
		\label{th: particular_case}
		Let $X$ be a smooth projective threefold such that $K_X\sim 0$, and $\X/S$ a  smooth projective family of varieties. Then the subsets $\Iso_X(\X/S),\Iso^1_X(\X/S)\subseteq S(\Ko)$ are constructible.
	\end{theorem}
	\begin{proof}
		Without loss of generality, we can assume that the set $\Iso_X^1(\X/S)$ is dense in $S$. Moreover, we may assume that for every closed point $s\in S$ the fiber $\X_s$ is a K-torsion threefold (cf. \cref{lemma: klt_0_families}). Hence, $\X/S$ is a lc-trvial fibration (see \cref{def: lc-trivial}).
		
		To be definite, we assume that $X = \X_{s_0}$ for some point $s_0 \in \Iso_X(\X/S)$. Choose an ample line bundle $\A$ on $\X/S$, and put $A = \A_{s_0}$. According to \cite[Corollary 2.3.5]{CMP17}, there is a polarization $Q$ on all of $V_\Zi = H^3(X, \Q)\cap H^3(X, \Zi)$ associated to the cup product on $H^3(X, \Zi)$ and the Lefschetz decomposition of $H^3(X, \Q)$ relative to $A$.
		The data $V = (V_\Zi, F^\bullet V_\Ko, Q)$ is a polarized integral Hodge structure of weight $3$, where $F^\bullet V_\Ko$ is the natural Hodge filtration on $V_\Ko = H^3(X, \Ko)$.

		In the standard way \cite{Gri70}, we define a variation $\V =  (\V_\Zi, F^\bullet \V_{\Os^\an}, \Qp)$ of integral polarized Hodge structures of weight $3$ associated with $\V_{\Os^\an} = (R^3f_* \underline{\Ko})\otimes \mathcal{O}_{S^\an}$, and our choice of an ample line bundle $\A$ on $\X/S$. Set $\Gamma^a = \text{Aut}(V_\Zi, Q)$. Let $\Phi: S^\an\to \Gamma^a\backslash\D$ be the associated period map. By \cite[Corollary 4.12]{Kol89}, all fibers $\X_{s}$ over the points $s\in \Iso^1(\X/S)$ have isomorphic (unpolarized) integral Hodge structures. Hence $\Phi(\Iso^1_{{X}}(\X/S))$ is a finite set by \cite[Lemma 2.22]{BFMT25}. Then, it follows from algebraicity of period maps (cf. \cite{BBT23}) that $\Phi$ is constant because $\Iso^1_{{X}}(\X/S)$ is dense in $S$. In particular, the deepest non-zero piece $H^{3,0}(\X_s)$ of the Hodge filtration on $H^3(\X_s, \Ko)$ is constant for all $s\in S^\an$. This implies that the Hodge line bundle $f_*\omega_{\X/S}$ is torsion (cf. \cref{lemma: hodge-trivial}).
		According to \cite[Theorem 4.7]{Amb05}, there exists a dense open subset $U\subseteq S$ such that $\X_s \approx X$ for every closed point $s\in U$. By Noetherian induction on closed subsets of $S$, there exist varieties $X_i$, and a finite decomposition of the base $S(\Ko) = \bigcup_i S_i$ into constructible sets $S_i$ such that for each index $i$ we have $\X_s \approx X_i$ for all $s \in S_i$. This concludes the proof. 
	\end{proof}
	
	\begin{definition}
		Suppose $(X, B)$ is a pair. Then the $\R$-divisor $B$ can be uniquely written as a difference of two effective $\R$-divisors $B_{\ge 0}$ and $B_{\le 0}$ without common components. We say that $N = \lceil B_{\le 0} \rceil$ is the \textit{negative part} of $B$, and $F = B - \lfloor B \rfloor$ is the \textit{fractional part} of $B$. In general, $F$ and $N$ may have common components. Then we can write
		$
			B = \lfloor B_{\ge 0} \rfloor + F - N.
		$
	\end{definition}
	
	\begin{construction}
		\label{const: covering_trick}
		Let $(X, B_X)$ be a klt log pair such that $B_X$ is a $\Q$-divisor and $K_X + B_X \sim_\Q 0$. Consider a log resolution $g: (Y, B_Y) \to (X ,B_X)$, where $K_Y + B_Y = g^*(K_X + B_X)$. Write $B_Y = F_Y-N_Y$, where $N_Y$ (resp. $F_Y$) is the negative (resp. fractional) part of~$B_Y$. Suppose $m\in \Zi_{>0}$ is the minimal integer such that $mF$ is integral, and $m(K_Y + B_Y)\sim~0$. Set $L = \mathcal{O}_Y(N_Y-K_Y)$. A choice of the section $s\in H^0(Y, L^{\otimes m})$ such that $mF_Y = (s = 0)$ determines a normalized cyclic cover $\pi:~(Z, B_Z) \to (Y, B_Y)$ of degree $m$ ramified over $F$. As always, $K_Z + B_Z = \pi^*(K_Y + B_Y)$. We shall say that $\pi: (Z, B_Z) \to (Y, B_Y)$ is an \textit{index-1 cover}. Moreover, $\pi_* \omega_Z \cong \bigoplus_{i=0}^{m-1}\omega_Y\otimes L^{\otimes i}(-\lfloor iF_Y \rfloor)$ (cf. \cite[Section 8.10]{Kol07}). In particular, $H^0(Y, \mathcal{O}_Y(N_Y))$ is naturally a direct summand of $H^0(Z, \omega_Z)$. The latter is the deepest non-zero piece of the Hodge filtration on $H^{\dim Z}(Z, \Ko)$.
	\end{construction}
	\begin{remark}
		\label{remark: section_independence}
		A normalized cyclic cover $\pi: (Z, B_Z) \to (Y, B_Y)$ is independent of the choice $s\in H^0(Y, L^{\otimes m})$ up to isomorphism. Indeed, for another section $s^\prime \in H^0(Y, L^{\otimes m})$ such that $mF = (s^\prime = 0)$ we have $s^\prime = \lambda s$ for some $\lambda\in \Ko^\times$  by \cite[Chapter II, Proposition 7.7]{Har77}. This determines an isomorphism of $\Os_X$-algebras $\mathcal{A} = \oplus_{i=0}^{m-1} L^{-i}$ and $\mathcal{A}^\prime = \oplus_{i=0}^{m-1} L^{-i}$ with different multiplication laws corresponding to our choices $s, s^\prime \in H^0(Y, L^{\otimes m})$ (see \cite[Definition 2.50]{KM98}). As a result, the schemes $\text{Spec}_Y \mathcal{A}, \text{Spec}_Y \mathcal{A}^\prime$ (and their normalizations) are isomorphic over $Y$.
	\end{remark}
	\begin{remark}
		\label{remark: mixed_are_pure}
		Since $(Y, B_Y)$ is log smooth, $Z$ has (at worst) quotient singularities (see \cite[Section 8.10.4]{Kol07}). Then the canonical mixed Hodge structure on $H^n(Z, \Q)$ is pure by \cite{Ste76}. 
	\end{remark}

	\begin{lemma}
		\label{lemma: invariance_of_deep}
		Suppose $Z_1, Z_2$ are birationally equivalent normal projective varieties with (at worst) quotient singularities. Let $\alpha_i \colon {Z}\to Z_i$ be a common resolution of singularities, $d = \dim Z$. Then the natural morphisms\footnote{For a given proper variety $X$ the torsion-free part of its cohomology group $H^n(X, \Zi)$ we denote by $H^n(X, \Zi)_\textnormal{tf}$.} $\alpha_i^*\colon H^d(Z_i, \Zi)_\textnormal{tf}\to H^d(Z, \Zi)_\textnormal{tf}$ of pure integral Hodge structures are injective. Furthermore, the morphisms $\alpha_1, \alpha_2$ induce isomorphisms:
		\begin{equation*}
			\gr^d_F H^d(Z_1, \Ko) \cong \gr^d_F H^d({Z}, \Ko) \cong \gr^d_F H^d(Z_2, \Ko)
		\end{equation*}
	\end{lemma}
	\begin{proof}
		Due to Deligne \cite[Theorem 11.1.19]{Max19}, the induced morphisms $(\alpha_i\otimes id_{\Q})^*\colon H^n(Z, \Q)\to H^n({Z}, \Q)$ of (pure) rational Hodge structures are injective for all  $0\le n \le 2d$. Hence,  the morphisms $\alpha_i^*\colon H^n(Z_i, \Zi)_\textnormal{tf}\to H^n({Z}, \Zi)_\textnormal{tf}$ are injective as well. Next, the natural morphisms of (pure) rational Hodge structures $H^n(Z_i, \Q) \to IH^n(Z_i, \Q)$ between singular and intersection cohomology groups are isomorphisms for all $0\le n \le 2d$  (see \cite[Section 11.3]{Max19}). In particular, $\gr^d_F H^d(Z_i, \Ko) \cong \gr^d_F IH^d({Z_i}, \Ko)$ for $i=1,2$. Due to \cite[Proposition 2.1]{KLS22},  the morphisms $\alpha_1, \alpha_2$ induce canonical isomorphisms $\gr^d_F IH^d(Z_1, \Ko) \cong \gr^d_F H^d({Z}, \Ko) \cong \gr^d_F IH^d(Z_2, \Ko)$. Then the lemma follows.
	\end{proof}
	\begin{remark}
		More generally {\cite[Theorem 0.5]{AH25}}, there are natural isomorphisms $\gr^n_F H^n(Z_i, \Ko) \cong \gr^n_F IH^n({Z_i}, \Ko) \cong \gr^n_F H^n(Z, \Ko)$ for all $0\le n\le 2d$ if $Z_i$ has (at worst) rational singularities.
	\end{remark}
	
	The following proposition is the central result of this section.
	\begin{proposition}[Birational invariance of periods]
		\label{lemma: bir_ivariant}
		Let $(X_1, B_{X_1})$ and $(X_2, B_{X_2})$ be 0-equivalent projective klt log pairs such that $B_{X_1}, B_{X_2}$ are $\Q$-divisors, and $K_{X_i} + B_{X_i} \sim_\Q 0$. Let $g_i^\circ: (Y_i, B_{Y_i}) \to (X_i, B_{X_i})$ be log resolutions, and $\pi_i: (Z_i, B_{Z_i})\to (Y_i, B_{Y_i})$ be index-1 coverings. Suppose $g_i: (Y, B_Y) \to (Y_i, B_{Y_i})$ is a common log resolution of $(Y_1, B_{Y_1})$ and $(Y_2, B_{Y_2})$. Let $\pi: (Z, B_Z)\to (Y, B_Y)$ be an index-1 cover. Then there exist birational contractions $\alpha_i: (Z, B_Z) \to (Z_i, B_{Z_i})$ making the following diagram commutative:
		$$
		\xymatrix{
			&(Z, B_Z)\ar@{->}[dl]_{\alpha_1} \ar@{->}[dr]^{\alpha_2} \ar@{->}[dd]^{\pi}&\\
			(Z_1, B_{Z_1})\ar@{->}[dd]_{\pi_1}  & & (Z_2, B_{Z_2})\ar@{->}[dd]^{\pi_2}\\
			&(Y, B_Y)\ar@{->}[dl]_{g_1}\ar@{->}[dr]^{g_2}&\\
			(Y_1, B_{Y_1})\ar@{->}[d]_{g_1^\circ} \ar@{<-->}[rr]&&(Y_2, B_{Y_2})\ar@{->}[d]^{g_2^\circ}\\
			(X_1, B_{X_1})\ar@{<-->}[rr]&&(X_2, B_{X_2}).
		}
		$$
		
		Also, the natural injective morphisms $\alpha_i^*: H^d(Z_i, \Zi)_\textnormal{tf}\to H^d(Z, \Zi)_\textnormal{tf}$ of pure integral Hodge structures induce isomorphisms for $d = \dim Z$:
		\begin{equation*}
			\gr^d_F H^d(Z_1, \Ko) \cong \gr^d_F H^d(Z, \Ko) \cong \gr^d_F H^d(Z_2, \Ko).
		\end{equation*}
		Moreover, under the natural maps from \cref{const: covering_trick}, we obtain the following:
		$$
		\xymatrix{
			H^0(Z_1, \omega_{Z_1}) \ar@{=}[r]& H^0(Z, \omega_Z)\ar@{=}[r] & H^0(Z_2, \omega_{Z_2})\\
			H^0(Y_1, \Os_{Y_1}(N_{Y_1}))\ar@{=}[r]\ar@{^{(}->}[u] &H^0(Y, \Os_{Y}(N_Y))\ar@{=}[r]\ar@{^{(}->}[u] &H^0(Y_2, \Os_{Y_2}(N_{Y_2})).\ar@{^{(}->}[u]
		}
		$$
		
	\end{proposition}
	\begin{proof}
		First let us construct natural birational maps $\alpha_i: Z \dashrightarrow Z_i$. Let $L = \Os_Y(N_Y - K_Y), \ L_i = \Os_{Y_i}(N_{Y_i} - K_{Y_i})$ be invertible sheaves as in \cref{const: covering_trick}. Let $s\in H^0(Y, L^{\otimes m})$ be the section inducing the cover $\pi: (Z, B_Z) \to (Y, B_Y)$, ramified over $F_Y$. Since the pairs $(X_1, B_{X_1}), (X_2, B_{X_2})$ are crepant birationally equivalent, we obtain that $(g_i)_*L = L_i$, and $(g_i)_*F_Y = F_{Y_i}$. Hence, the section $s$ induces sections $s_i\in H^0(Y, L_i^{\otimes m})$ such that $(s_i = 0) = mF_{Y_i}$. According to \cref{remark: section_independence}, we can assume that the index-1 covers $\pi_i: (Z_i, B_{Z_i})\to (Y_i, B_{Y_i})$ correspond to the sections $s_i$. Let $U\subseteq Y$ be the largest open subset such that $g_i|_U: U \to g_i(U)$ is an isomorphism for $i = 1, 2$. By construction, the index-1 cover $\pi|_{\pi^{-1}(U)}: \pi^{-1}(U) \to U$ is isomorphic to the index-1 cover $\pi_i|_{\pi_i^{-1}(g_i(U))}: \pi_i^{-1}(g_i(U) \to g_i(U)$ over $g_i(U)$. By $\alpha_i$ denote the constructed isomorphisms $\pi^{-1}(U) \to \pi_i^{-1}(g_i(U)$. As a result, both $\alpha_i: Z \dashrightarrow Z_i$ are natural birational maps making the diagram (stated in the lemma) commutative.
		
		Now let us prove that the birational maps $\alpha_i: Z\dashrightarrow Z_i$ are morphisms. Notice that $\alpha_i$ is a birational contraction because $g_i$ is a birational contraction, and $\pi_i, \pi$ are finite morphisms. Let $A_i$ be a general ample prime divisor on $Y_i$. The divisors $D_i = g_i^*(A_i), \ D_i^\prime = \pi^*(D_i)$ are big and nef on $Y, Z$ respectively, and $A_i^\prime = \pi_i^* (A_i)$ is ample on $Z_i$. It follows easily from \cite[Definition 3.6.1]{BCHM10} that $g_i$ is a $D_i$-nonpositive morphism. We claim that the map $\alpha_i: Z \dashrightarrow Z_i$ is $D^\prime_i$-nonpositive. In our case, this is equivalent to $(K_{Z} + B_Z + \epsilon D^\prime_i)$-nonpositivity of $\alpha_i$ for any $\epsilon\in (0,1)\cap \Q$ as $K_Z + B_Z \sim 0$. It is clear from the definition of $\alpha_i$ that a prime divisor $E^\prime_i$ on $Z$ is exceptional for $\alpha_i$ if and only if $\pi(E^\prime_i)$ is exceptional for $g_i$. Let $E^\prime_i \subset Z$ be an exceptional prime divisor for $\alpha_i$, and $r\le m = \deg \pi$ be the ramification index of $\pi$ along $E^\prime_i$. Set $E_i = \pi (E^\prime_i)$. In the proof of \cite[Proposition 5.20]{KM98} a local computations shows that $$a(E^\prime_i; Z, B_{Z} + \epsilon D^\prime_i) + 1 = r\left(a(E_i; Y, B_Y + \epsilon D_i) + 1\right).$$ The latter equality together with $(K_{Y} + B_Y + \epsilon D_i)$-nonpositivity of $g_i$ implies that the birational contraction $\alpha_i$ is $(K_{Z} + B_Z + \epsilon D^\prime_i)$-nonpositive by \cite[Lemma 3.6.3]{BCHM10}. Let $W_i$ be a common log resolution of $(Z, B_Z + \epsilon D^\prime_i)$ and $(Z_i, B_{Z_i} + \epsilon A^\prime_i)$, and $p_i, q_i$ be the induced morphisms:
		$$
		\xymatrix{
			&W_i\ar@{->}[dl]_{p_i}\ar@{->}[dr]^{q_i}&\\
			Z\ar@{-->}[rr]^{\alpha_i}&&Z_i.
		}
		$$
		
		By the definition of nonpositive birational contractions, we have that $$E_i := p^*_i(K_Z + B_Z + \epsilon D_i^\prime) - q^*_i(K_{Z_i} + B_{Z_i} + \epsilon A_i^\prime) \sim_\Q \epsilon(p_i^*D^\prime - q_i^*A_i^\prime)$$ is an effective and $q_i$-exceptional $\Q$-divisor. Furthermore,  since $D^\prime$ is big and nef, $E_i$ is $q_i$-nef. By the negativity lemma \cite[Lemma 3.6.2]{BCHM10}, we obtain that $E_i =0$. In particular, the $0$-pairs $(Z, B_Z)$ and $(Z_i, B_{Z_i})$ are 0-equivalent, and $p_i^*(D^\prime) = q_i^*(A_i^\prime)$. Assume that there exists a point $\star_i\in Z$ where the map $\alpha_i$ is not defined. Then there is a curve $C_i\subset W_i$ such that $p_i(C_i) = \star_i$, and $q_i(C_i)$ is a curve. But, $0 =p_i^*(D^\prime) \cdot C_i =  q_i^*(A_i^\prime) \cdot C_i = A_i^\prime \cdot q_i(C_i) > 0$ that is an absurd. Therefore, we proved that the natural maps $\alpha_i$ are birational contractions $\alpha_i: (Z, B_Z) \to (Z_i, B_{Z_i})$ of crepant birationally equivalent projective klt $0$-pairs.
		
		Finally, we prove the statement about Hodge structures. Let $\nu: \widetilde{Z}\to Z$ be a resolution of singularities. Then the morphisms of pure integral Hodge structures $\nu^*\colon H^d(Z, \Zi)_\text{tf} \to H^d(\widetilde{Z}, \Zi)_\text{tf}$ and $(\alpha_i\circ \nu )^*\colon H^d(Z_i, \Zi)_\text{tf} \to H^d(\widetilde{Z}, \Zi)$ are injective by \cref{lemma: invariance_of_deep}. This imply injectivity for the maps $\alpha_i^*\colon H^d(Z_i, \Zi)_\textnormal{tf}\to H^d(Z, \Zi)_\textnormal{tf}$. Hence, the isomorphisms $\gr^d_F H^d(Z_1, \Ko) \cong \gr^d_F H^d(Z, \Ko) \cong \gr^d_F H^d(Z_2, \Ko)$ follow from \cref{lemma: invariance_of_deep}.
	\end{proof}

	\subsection{Lc-trivial fibrations}
	We follow an exposition in \cite[Section 6]{BFMT25}, and refer a reader to \textit{loc. cit.} for further discussions of birational geometry of lc-trivial fibrations.
	\begin{definition}
		\label{def: lc-trivial}
		Let $(\X, \B)$ be a log pair, $\B$ is a $\Q$-divisor, and $S$ a normal variety. A proper contraction $f: \X \to S$ is called \textit{lc-trivial} if
		\begin{enumerate}
			\item 
			$(\X_\eta, \B_\eta)$ is a lc log pair, where $\eta$ is the generic point of $S$.
			\item 
			$K_\X + \B \sim_{\Q, S} 0$, i.e. $K_\X + \B \sim_\Q f^*L$ for some $\Q$-Cartier $\Q$-divisor $L$ on $S$.
			\item 
			$\textnormal{rk}\, f_*\Os\left(\lceil \mathbf{A}^*(\X, \B)\rceil\right) = 1$,
		\end{enumerate}
		where $\mathbf{A}^*(\X, \B)$ is the b-divisor defined by taking its trace on any birational model $g: \Y\to \X$, that is, $\mathbf{A}^*(\X, \B)_{\Y} := K_{\Y} - g^*(K_\X + \B) + \sum_{a(E; \X, \B) = 1} E.$
	\end{definition}
	\begin{remark}
		The property $(3)$ holds if the generic fiber $(\X_\eta, \B_\eta)$ is klt, and $\B_\eta$ is effective (cf. \cite[Section 8.4]{Kol07}). It guarantees that a Hodge vector bundle is a line bundle (cf. \cref{const: main}). 
	\end{remark}

	\begin{definition}
		\label{def: CY-VHS}
		Let $S$ be a regular variety. Let $\V = (\V_\Zi, F^\bullet\V_{\Os^\an})$ be a polarizable integral variation of pure Hodge structures ($\Zi$-VHS) on $S^\an$. We say $\V$ is a \textit{Calabi--Yau variation} if the deepest non-zero part $F^n\V_{\Os^\an}$ of the Hodge filtration has rank one. In this case, we call $\V$ a polarizable CY $\Zi$-VHS, and $\Ho_S = F^n \V_{\Os^\an}\cong \gr^n_F\V_{\Os^\an}$ the Hodge line bundle.
	\end{definition}

	\begin{construction}
		\label{const: main}
		Let $f_\X: (\X, \B_\X) \to S$ be a projective lc-trivial fibration of relative dimension $d\in \Zi_{>0}$. We shall assume that $(\X, \B)$ is a klt log pair, and $(\X/S, \B)$ is a projective elementary family. Consider a log resolution $g: (\Y, \B_\Y) \to (\X, \B_\X)$. We assume that $(\Y/S, \B_\Y)$ is a projective elementary family, and $(\Y, \B)$ is log smooth over $S$. As in \cref{const: covering_trick}, we write $\B_\Y = F_\Y - N_\Y$, and consider $\pi: (\Z, \B_\Z)\to (\Y, \B_\Y)$ an index-1 cover of degree $m\in \Zi_{>0}$. Choose a $\mu_m$-equivariant resolution of singularities $\nu: \widetilde{\Z} \to \Z$:
		$$
		\xymatrix{
			(\Y, \B_\Y)\ar@{->}[d]_{g} \ar@/_2.5pc/[dd]_{f_\Y} &(\Z, \B_\Z)\ar@{->}[l]_{\pi} \ar@{->}[ddl]^(0.4){f_\Z}&(\widetilde{\Z}, \B_{\widetilde{\Z}})\ar@{->}[l]_(0.5){\nu}  \ar@{->}[ddll]^(0.4){f_{\widetilde{\Z}}}\\
			(\X, \B_\X) \ar@{->}[d]_{f_\X} & &\\
			S & &.
		}
		$$
		
		We shall assume that the induced morphism $f_{\widetilde{\Z}}$ is smooth. The sheaf $(\pi\circ\nu)_*\omega_{\widetilde{\Z}}$ is $\mu_m$-equivariant, and it admits a decomposition into $\mu_m$-isotypic components:
		\begin{equation*}
			(\pi\circ\nu)_*\omega_{\widetilde{\Z}} = \pi_* \omega_{\Z} \cong \bigoplus_{i=0}^{m-1} \pi_*(\omega_\Z)_{\chi^i} \cong \bigoplus_{i=0}^{m-1}\omega_{\Y}\otimes \Li^{\otimes i}\left(-\lfloor i F_\Y\rfloor\right),
		\end{equation*}
		where $\chi$ is a generator of the character group $\hat{\mu}_m = \text{Hom}(\mu_m, \Ko^\times)$. In particular, we get the direct summand for $i=1$:
		\begin{equation*}
			\Os_{\Y}(N_\Y) \cong (\pi\circ \nu)_*\left(\omega_{\widetilde{\Z}}\right)_\chi \hookrightarrow (\pi\circ \nu)_* \omega_{\widetilde{\Z}}.
		\end{equation*}
		Let $\V_\Zi = \left(R^d(f_{\widetilde{\Z}})_*\underline{\Zi}\right)_\textnormal{tf}$, and $\V_\Ko = \V_\Zi\otimes \Ko$. By $\V_{\chi^i}$ we denote the $\chi^i$-isotypic component of $(\V_\Ko, F^\bullet \V_{\Os^\an})$, where $i = 1, 2, \dots, m$. Then after shifting the Hodge filtration $$\V_\chi(-2, 2)\oplus \left(\bigoplus_{i = 0,2,3,\dots, m-2} \V_{\chi^i}\right)\oplus \V_{\chi^{m-1}}(2, -2)$$ we a obtain polarizable CY $\Zi$-VHS $\V$ of Hodge-level $d+4$ whose underlying local system is $\V_\Zi$.
		By construction, the Hodge line bundle for $\V$ is $\Ho_S = (f_\Y)_*\Os_\Y(N_\Y)$.
	\end{construction}

	\begin{definition}
		\label{def: transcendental}
		Let $S$ be a regular variety. Let $\V = (\V_\Zi, F^\bullet\V_{\Os^\an})$ be a polarizable $\Zi$-VHS on $S^\an$, and $F^n \V_{\Os^\an}$ be the deepest non-zero part of the Hodge filtration. Put $\V_\Q = \V_\Zi\otimes \Q$. By $(\V^\text{tr}_\Q, F^\bullet\V_{\Os^\an}^\text{tr})$ we denote the \textit{transcendental part of $\V$}, that is, a (unique) smallest polarizable rational Hodge substructure of $(\V_\Q, F^\bullet\V_{\Os^\an})$ for which $F^n \V_{\Os^\an}^\text{tr} = F^n \V_{\Os^\an}$. Note that $(\V^\text{tr}_\Q, F^\bullet\V_{\Os^\an}^\text{tr})$ is simple if $\V$ is a CY $\Zi$-VHS. We say $\V^\text{tr} = (\V^\text{tr}_\Zi, F^\bullet\V_{\Os^\an}^\text{tr})$ is the \textit{transcendental Hodge lattice for} $\V$, where $\V^\text{tr}_\Zi = \V^\text{tr}_\Q\cap \V_\Zi$.
	\end{definition}
	\begin{definition}
		\label{def: CY_period_map}
		Let $S$ be a regular variety, and $s_0\in S$ be a closed point. Let $\V = (\V_\Zi, F^\bullet\V_{\Os^\an})$ be a polarizable CY $\Zi$-VHS on $S^\an$. Choose a polarization $\Qp$ on $\V^\text{tr}$. Put $V_\Zi = (\V_\Zi^\text{tr})_{s_0}$, $Q = \Qp_{s_0}$. Let $\Gamma^a = \text{Aut}(V_\Zi, Q)$ be the \textit{arithmetic monodromy group}, that is, a subgroup of $\text{GL}(V_\Zi)$ preserving the polarization $Q$. By $\Phi^\text{CY}: S^\an\to \Gamma^a\backslash \D^\text{CY}$ we denote the period map associated with $\V^\text{tr}, \Qp$, and $s_0$. By \cite[Lemma-Definition 4.6.3]{CMP17}, $\Phi^\text{CY}$ is a holomorphic map between analytic spaces. We shall say that $\Phi^\text{CY}$ is a \textit{CY period map} associated with the polarizable CY $\Zi$-VHS $\V$.
	\end{definition}
	\begin{remark}
		It is readily seen that CY period maps are independent of the choice $s_0\in S$ up to isomorphism. However, CY period maps depend on the choice of a polarization $\Qp$ on $\V^\text{tr}$ apriori.
	\end{remark}
	
	\begin{lemma}
		\label{lemma: hodge-trivial}
		Let $S$ be a regular variety, and $\V$ be a polarizable CY $\Zi$-VHS on $S^\an$. Let $\Ho$ be the Hodge line bundle, $\Phi^\CY: S^\an\to \Gamma^a\backslash\D^\CY$ be a CY period map associated with $\V$. If $\Phi^\CY(S^\an) = \{\star\}$, then $\Ho\sim_\Q 0$.
	\end{lemma}
	\begin{proof}
		Let $\Gamma\le \Gamma^a$ be the monodromy group, and $\xymatrix@C=1.5pc{p: \Gamma \backslash \D^\CY\ar@{->>}[r]& \Gamma^a \backslash \D^\CY}$ be the projection. Note that the CY period map $\Phi^\CY: S^\an \to \Gamma^a \backslash \D^\CY$ naturally factors through the period map $\Phi^\CY_0: S^\an \to \Gamma \backslash \D^\CY$, and $\Phi^\CY_0(S^\an)\subseteq p^{-1}(\star)$. From algebraicity of period maps \cite[Theorem 1.1]{BBT23} it follows that $\Phi^\CY_0(S^\an)$ is an irreducible algebraic set. Hence, $\Phi^\CY_0(S^\an)$ is a single point provided $p^{-1}(\star)$ is an at most countable set. Let $F^n \V_{\Os^\an}$ be the deepest non-zero part of the Hodge filtration. Since the period map $\Phi^\CY_0$ is constant, we conclude that $\text{gr}^n_F \V_{\Os^\an}^\text{tr}$ is flat (see \cite[Lemma 13.1.8]{CMP17}). Then $\text{gr}^n_F \V_{\Os^\an}^\text{tr}$ is torsion according to \cite{Del71}. By \cref{def: transcendental}, $\Ho_S \cong\text{gr}^n_F \V_{\Os^\an} = \text{gr}^n_F \V_{\Os^\an}^\text{tr}\sim_\Q 0$.
	\end{proof}
	
	\begin{definition}
		Let $f: (\X, \B) \to S$ be a lc-trivial fibration. \textit{The boundary divisor} $B_S$ is the $\Q$-divisor on $S$ whose coefficient along a prime divisor $P$ is given by
		\begin{equation*}
			\begin{gathered}
				\text{ord}_P(B_S) = 1 - \text{lct}_{\eta_P}(\X, \B; f^*P),\\
				\text{lct}_{\eta_P}(\X, \B; f^*P) = \sup \{t\in \R: (\X, \B + tf^*(P)) \text{ is lc over the generic point $\eta_P$ of $P$ }\}.
			\end{gathered}
		\end{equation*}
	\end{definition}
	
	\begin{definition}[Prokhorov--Shokurov' moduli divisor]
		Let $f: (\X, \B) \to S$ be a lc-trivial fibration, and $\eta\in S$ the generic point. Suppose $m\in \Zi_{>0}$ is the minimal integer such that $m\B_\eta$ is integral and $m(K_{\X_\eta} + \B_\eta) \sim 0$. \textit{A formal moduli divisor} $M_S$ is a $\Q$-divisor on $S$ satisfying the following property: there exists a rational function $\phi\in k(\X)^\times$ such that
		\begin{equation*}
			m(K_\X + \B) = m\left(f^*(K_S + B_S + M_S)\right) + \text{div}(\phi).
		\end{equation*}
	\end{definition}
		Note that the formal moduli divisor $M_S$ is defined up to $m$-linear equivalence, choices of $K_S, K_\X, \phi$. Nevertheless, the divisor $M_S$ captures variation for klt fibers of an lc-trivial fibration as the following theorem shows. In fact, $M_S \sim_\Q \Ho_S$ (cf. \cite[Lemma 5.2]{Amb04}).
	\begin{theorem}[{\cite[Theorem 4.7]{Amb05}}]
		\label{th: Ambro_isotriviality}
		Let $f: (\X, \B)\to S$ be a lc-trivial fibration between proper varieties satisfying the following conditions:
		\begin{itemize}
			\item 
			$(\X_\eta, \B_\eta)$ is a projective klt log pair, and $ \B_\eta$ is an effective $\Q$-divisor.
			\item 
			$B_S = 0$.
			\item 
			$M_S \sim_\Q 0$.
		\end{itemize}
		Then there exists a finite Galois covering $\tau: S^\prime \to S$, which is \'etale in codimension 1, a non-empty open subset $U\subseteq S^\prime$, a log pair $(X, B)$, and an isomorphism $(\X, \B)\times_S U \approx (X, B)\times_\Ko U$ over $U$.
	\end{theorem}

	One may wonder if \cref{th: Ambro_isotriviality} holds for lc-trivial fibrations $f: (\X, \B)\to S$ such that $(\X_\eta, \B_\eta)$ is not a klt log pair. The following examples shows this is not the case.
	\begin{example}
		Let $S = \Pp_\Ko \smallsetminus \{0, 1, \infty\}$, $\X = \Pp^1_\Ko\times_\Ko S$, and $\Delta \subset \Pp^1_\Ko\times_\Ko S$ be the restriction of the diagonal. Suppose $$\B = a\Delta + \B^\prime, \ \B^\prime = a_0(\{0\}\times_\Ko S) + a_1(\{1\}\times_\Ko S) + a_{\infty}(\{\infty\}\times_\Ko S)$$
		is an $\Q$-divisor on $X$. Assume that all coefficients $a, a_0, a_1, a_\infty$ are non-zero. Then the projective elementary family $\text{pr}: (\X, \B)\to S$ is not \'etale trivial at any closed point $s_0\in S$, that is, for any \'etale neighborhood $(U, s_0)$ the family $(\X_U/U, \B_U)$ is not trivial over $U$.
		
		If $(a, a_0, a_1, a_\infty) = (1, \frac{1}{3}, \frac{1}{3}, \frac{1}{3})$ then $\text{pr}: (\X, \B)\to S$ is a lc-trivial fibration such that the generic fiber is lc, but not klt. It is clear that $B_S = 0$. Let $M_S$ be a formal moduli divisor. Then the induced lc-trivial fibration $\text{pr}|_\Delta: (\Delta, \B^\prime|_{\Delta}) \to S$ has the same moduli divisor $M_S$. Note $\B^\prime|_{\Delta} =0$, and $\text{pr}|_\Delta: \Delta\to S$ is an isomorphism. Hence $M_S \sim_\Q 0$.
		
		In contrast, if $(a, a_0, a_1, a_\infty) = (\frac{1}{2}, \frac{1}{2}, \frac{1}{2}, \frac{1}{2})$ then $M_S\nsim_\Q 0$ by \cref{th: Ambro_isotriviality}. Let us show this directly. Suppose $\pi: (\Z, 0)\to (\X, \B)$ is an index-1 cover. For every point $s\in S(\Ko)$ the induced cover $(\Z_s, 0) \to (\X_s, \B_s)$ (locally) given by the equation $y^2 = x(x-1)(x-s)$, where $x$ is a local coordinate on $\Pp^1_\Ko\times s$. Hence, $(\Z, 0) \to S$ is a projective family of elliptic curves with non-constant $j$-invariant. Hence $M_S\nsim_\Q 0$.
	\end{example}

	\section{Main results}
	\subsection{K-torsion case}\label{subsec: CY}
	Let $(\X/S, \B)$ be an elementary family, and $(X,B)$ be a pair. Suppose $s\in S$ is a point, $\kappa(s)$ is its residue field, and $\overline{\kappa(s)}$ is an algebraic closure for  $\kappa(s)$. By $(\X_{\overline{s}}, \B_{\overline{s}})$ we denote the geometric fiber over $s$, that is, $(\X_s, \B_s)\times_{\kappa(s)} \overline{\kappa(s)}$. We write $(\X_{\overline{s}}, \B_{\overline{s}}) \overset{\textnormal{small}}{\longdashleftrightarrow } (X, B)\times_\Ko \overline{\kappa(s)}$ assuming that $(\X_{\overline{s}}, \B_{\overline{s}})$ is a log pair $0$-equivalent (\cref{def: 0-class})  to $(X, B)\times_\Ko \overline{\kappa(s)}$, and the birational map $\X_{\overline{s}} \longdashleftrightarrow X\times_\Ko \overline{\kappa(s)}$ is small. Similarly, we say that $(\X_{\overline{s}}, \B_{\overline{s}})$ is \textit{a wlc model for} $(X, B)\times_\Ko \overline{\kappa(s)}$ if $(\X_{\overline{s}}, \B_{\overline{s}})$ is a wlc model (\cref{def: 0-pair}), and $(\X_{\overline{s}}, \B_{\overline{s}})$ is $0$-equivalent to $(X, B)\times_\Ko \overline{\kappa(s)}$. Define the following sets:
	\begin{equation*}
		\begin{gathered}
			\text{Iso}_{(X, B)}(\X/S, \B) =  \{s\in S: (\X_{\overline{s}}, \B_{\overline{s}})\approx (X, B)\times_\Ko \overline{\kappa(s)}\},\\
			\text{Iso}^1_{(X, B)}(\X/S, \B) =  \{s\in S: (\X_{\overline{s}}, \B_{\overline{s}}) \overset{\textnormal{small}}{\longdashleftrightarrow } (X, B)\times_\Ko \overline{\kappa(s)}\},\\
			\Wlc_{(X, B)}(\X/S, \B) = \{s\in s: (\X_{\overline{s}}, \B_{\overline{s}}) \text{ is a wlc model for }  (X, B)\times_\Ko \overline{\kappa(s)}\}.
		\end{gathered}
	\end{equation*}
	
	\begin{lemma}[Birational invariance of the transcendental Hodge lattice]
		\label{lem: bir_tr_lattice}
		Let $(X_1, B_{X_1})$ and $(X_2, B_{X_2})$ be 0-equivalent projective klt 0-pairs with $\Q$-coefficients,  $V_1$ and $V_2$  be polarizable CY $\Zi$-HS as in \cref{const: main}. Consider  the corresponding transcendental Hodge lattices $T_1$ and $T_2$ naturally polarized by the cup product. Then there is an isomorphism $T_1\approx T_2$ of polarized $\Zi$-HS.
	\end{lemma}
	\begin{proof}
		We shall use the notation of  \cref{lemma: bir_ivariant}. Then there are birational contractions $\alpha_i: Z\to Z_i$ inducing injective morphisms of $\Zi$-HS $\alpha_i^*: H^d(Z_i, \Zi)_\text{tf}\to H^d(Z, \Zi)_\text{tf} $ such that $\gr_F^d H^d(Z_1, \Ko)\cong \gr_F^d H^d(Z, \Ko) \cong \gr_F^d H^d(Z_2, \Ko)$. Let $V^\text{tr}$ (resp. $V^\text{tr}_i$) be the transcendental Hodge lattice for $H^d(Z, \Zi)_\text{tf}$ (resp. $H^d(Z_i, \Zi)_\text{tf}$). Then $\alpha_1^*V^\text{tr}_1\supseteq V^\text{tr}$ and $V^\text{tr}_1\subseteq (\alpha_1)_*V^\text{tr}$ according to \cref{def: transcendental}. This implies that $\alpha_1^*V^\text{tr}_1 = V^\text{tr}_s$. Similarly, $\alpha_2^*V^\text{tr}_2 = V^\text{tr}$. Recall that $V_i$ is constructed by shifting the Hodge filtration on $H^d(Z_i, \Zi)_\text{tf}$. Again by \cref{lemma: bir_ivariant},  we have $H^0(Y_i, \Os_{Y_i}(N_{Y_i})) = H^0(Y, \Os_{Y}(N_{Y}))$ which is the deepest non-zero piece of the Hodge filtration for $V_i$. Hence we obtain $\alpha_1^*T_1 = \alpha_2^* T_2$  by the same argument as above applied to $T_i$ and $V_i$.
		
		It remains to show the isomorphism $T_1\approx T_2$ preserves the cup product. Let $\nu_i: \widetilde{Z}_i\to Z_i$ be equivariant resolutions of singularities. Then there is an equivariant resolution $\nu: \widetilde{Z}\to Z$ with birational contractions $p_i: \widetilde{Z}\to \widetilde{Z}_i$. By the projection formula for the cap product $\cap$ with the Poincar\'e duality \cite[Theorem 3.30]{Hat02} we obtain $p_i^*a\cup p_i^* b = a\cup b$ for all $a,b\in H^d(\widetilde{Z}_i, \Zi)_\text{tf}$.
	\end{proof}
	
	\begin{definition}
		We say that a morphism $\tau: U\to S$ between regular varieties is an \textit{\'etale base change} if the image $\tau(U)$ is an open subset of $S$, and $\tau: U\to \tau(U)$ is an \'etale covering, that is, an \'etale, finite, surjective morphism. For an elementary family $(\X/S, \B)$ with $\B = \sum b_i \B_i$ we set $(\X, \B)\times_S U = (\X\times_S U, \sum b_i (\B_i\times_S U))$. It is clear that $(\X, \B)\times_S U$ is an elementary family over~$U$.
	\end{definition}
	
	\begin{theorem}
		\label{th: constructible}
		Let $(X, B)$ be a projective klt $0$-pair, $B$ an effective $\Q$-divisor, and $(\X/S, \B)$ a projective elementary family. Then the subsets $\Iso_{(X, B)}(\X/S, B), \Iso_{(X, B)}^1, \Wlc_{(X, B)}(\X/S, B)\subseteq S$ are constructible.
	\end{theorem}
	\begin{proof}
		Without loss of generality, we can assume that ${\Wlc^1_{(X, B)}(\X/S, \B)}$ is dense in $S$. Let $I$ be an index set such that $\{s_\al: \al\in I\} = \left(\Wlc_{(X, B)}^1(\X/S, \B)\right)(\Ko)$. Let $\mathcal{S}$ be the set of all closed subsets $S^\prime \subseteq S$ such that $\{s_\alpha\}_{\alpha\in I}\cap S^\prime$ is dense in $S^\prime$. Note that $\mathcal{S}$ is a well-founded set with respect to inclusion as $S$ is Noetherian. Then the set of all minimal elements for $\mathcal{S}$ is $\{\{s_\alpha\}: \al \in I\}$. We say that a property $P(S^\prime)$ holds for $S^\prime \in \mathcal{S}$ if the set $\Wlc_{(X, B)}^1(\X/S, \B)\cap S^\prime$ is constructible. We shall proceed by Noetherian induction on $\mathcal{S}$. Clearly, $P(\{s_\al\})$ holds for all $\al\in I$ (Noetherian induction base). Let $S^\prime\in \mathcal{S}$ be a non-minimal element of $\mathcal{S}$. Assume that $P(S^{\prime\prime})$ is true for all proper subsets $S^{\prime\prime}\subset S^\prime$ such that $S^{\prime\prime}\in \mathcal{S}$ (Noetherian induction step). We claim that $P(S^\prime)$ holds. To prove the claim, it suffices to show that there is an \'etale base change $U\to S^\prime$ such that $(\X_U, \B_U) = (\X, \B)\times_S U \approx (X, B)\times_\Ko U$ over $U$. Then the statement follows from Noetherian induction (cf. \cite[Chapter II, Exercise 3.16]{Har77}).
		
		 It is readily seen that the problem can be reduced to the case when $((\X_{S^\prime})/S^\prime, \B_{S^\prime})$ is a projective elementary family by Noetherian induction step. To simplify the notation, we write $(\X/S, \B)$ for the corresponding new family. Consider a $\Q$-factorial terminalization $(X_\text{trm}, B_\text{trm})$ for $(X, B)$ (cf. \cref{prop: flops_0-equivalent}). Let $B_\text{trm} = \sum_i b_i (B_\text{trm})_i$, choose any $0<\epsilon< \min\{1-b_i\}$. Then for every closed point $s\in\Wlc_{(X, B)}(\X/S,\B)$ the pair $(\X_s, \B_s)$ is an $\epsilon$-klt.  By \cref{lemma: klt_0_families}, we may assume that $(\X, \B)$ is a klt log pair, and $\text{klt}_\epsilon(\X/S, \B) = \nu_0(\X/S, \B) = S$. Then $f: (\X, \B)\to S$ is a projective lc-trivial fibration satisfying the assumptions of \cref{const: main} (after shrinking the base).
		
		First, we prove that the boundary divisor $B_S$ is zero. Recall that all scheme-theoretic fibers of $\X/S$ are integral, and $S$ is regular (see \cref{def: elementary_family}). Hence, for every prime divisor $P$ on $S$, the pullback $D = f^*(P)$ is a prime divisor, and it is Cartier. Since every prime component of $\B$ is horizontal over $S$, we obtain the inequality $\text{lct}_{\eta_P}(\X, \B; D) \le 1$. We claim that $\text{lct}_{\eta_P} (\X, \B; D) \ge 1$, that is, the pair $(\X, \B + D)$ is lc over the generic point $\eta_P$ of $P$. Indeed, by the inversion of adjunction \cite[Theorem 4.9]{Kol13}, we have that $(\X, \B + D)$ is lc near $D$ if and only if $(\overline{D}, \text{Diff}_{\overline{D}}\B)$ is lc, where $\overline{D}\to D$ is a normalization. Note that the induced fibration $\overline{D}\to P$ must factor through a normalization $\overline{P}\to P$. According to \cref{lemma: klt_0_families}, the induced lc-trivial fibration $(\overline{D}, \text{Diff}_{\overline{D}}\B)\to \overline{P}$ is klt over the generic point $\eta_P$. This implies $\text{lct}_{\eta_P} (\X, \B; D) \ge 1$, as claimed. Hence $B_S = 0$.
		
		Second, we show that $M_{\overline{S}}\sim_\Q 0$ for some compactification $S\subseteq \overline{S}$. Then the statement will follow from \cref{th: Ambro_isotriviality}. Let $g: (\Y, \B_\Y) \to (\X, \B_\X)$ be a log resolution. By the generic smoothness, we can assume that this log resolution satisfies the assumptions of \cref{const: main} after shrinking the base. Similarly, we suppose an index-1 cover $\pi: (\Z, \B_{\Z})\to (\Y, \B_\Y)$, and its equivariant resolution $\nu: (\widetilde{\Z}, \B_{\widetilde{Z}})\to (\Z, \B_{\Z})$ satisfy the assumptions of \cref{const: main}. Therefore, there is a polarizable CY $\Zi$-VHS $\V$ whose Hodge line bundle is $\Ho_S = (f_{\Y})_*\Os_{\Y}(N_\Y)$, where $N_\Y$ is the negative part of~$\B_\Y$. To be definite, choose a closed point $s_0\in \Wlc_{(X, B)}(\X/S, \B)$. Set $\Phi^\CY: S^\an \to \Gamma^a\backslash \D^\CY$ to be a CY period map associated with $\V^\text{tr}, s_0$, and the polarization $\Qp$ given by the cup product on $\widetilde{\Z}$. The fiber of $\Ho_S$ over a point $s\in S(\Ko)$ equals to $H^0(\Y_s , \Os_{\Y_s}((N_\Y)_s)$ where $(N_\Y)_s = N_{\Y_s}$ is the negative part of $(\B_\Y)_s = \B_{\Y_s}$. By $T$ we denote the transcendental Hodge lattice for $(X, B)$. We claim that $\text{rk}\, T = \text{rk}\, (\V^\text{tr})_s$ for a very general point $s\in S(\Ko)$.
		Assume the claim holds. Then for every closed point $s\in \Wlc_{(X, B)}(\X/S, \B)$ there is an isomorphism $(\V^{\text{tr}})_s\approx T$ of polarized $\Zi$-HS by \cref{lem: bir_tr_lattice}. Then $\Phi^\CY\left(\Wlc^1_{(X, B)}(\X/S, \B)\right) = \{\star\}$. The subset $(\Phi^\CY)^{-1}(\star)$ is an algebraic subset of $S^\an$ by \cite[Theorem 1.1]{BBT23}, and dense by assumption. Hence, $(\Phi^\CY)(S^\an) = \{\star\}$. After taking a SNC compactification $\overline{S}\supseteq S$ and extending the period map across the (open) locus $S_f\subseteq \overline{S}$ of finite monodromy \cite{Gri70}, we obtain $\Phi^\CY(S_f^\an) = \{\star\}$. Hence, $S_f=\overline{S}$ again by properness of period maps. By \cref{lemma: hodge-trivial}, we have $\Ho_S\sim_\Q 0$. The canonical bundle formula \cite[Theorem 6.28]{BFMT25} tells us that $\Ho_S\sim_\Q M_S$. Therefore, \cref{th: Ambro_isotriviality} implies that there exists an \'etale base change $U\to S$ such that $(\X, \B)\times_S U \approx (X, B)\times_\Ko U$ over $U$.
		
		Now, we prove the claim by contradiction. Put $\mathbb{V} = \V^\text{tr}$, and  $R= \text{Wlc}_{(X, B)}(\X/ \B)(\Ko)$. Let us remember that $R$ is dense in $S^\an$, and for every point $s\in R$ there is an isomorphism $(\mathbb{V}_s)^\text{tr}\approx T$ of polarized $\Zi$-HS. Suppose $\text{rk}\, T < \mathbb{V}_s$ for a very general point $s\in S^\an$.  In other words, $\text{Wlc}_{(X, B)}(\X/, \B)$ is included in the Hodge locus $\text{HL}(\mathbb{V}^\otimes)$. Then for every point $s\in R$ there is an orthogonal decomposition $(\mathbb{V}_s)_\Q = (\mathbb{V}_s)^\text{tr}_\Q\oplus U_s$ with respect to the polarization $\Qp_s$. The projection $P_s: \mathbb{V}_s\to (\mathbb{V}_s)^\text{tr} \subseteq (\mathbb{V}_s)_\Q$ is a Hodge class in $\text{End}_{\Q-\text{HS}}((\mathbb{V}_s)_\Q, (\mathbb{V}_s)_\Q)$. However, $\Qp_s(P_sv, t) = \Qp_s(v, t)\in \Zi$ for all $v\in \mathbb{V}_s$ and $t\in (\mathbb{V}_s)^\text{tr}$, that is, $P_s(\mathbb{V}_s)\subseteq \text{Hom}_\Zi((\mathbb{V}_s)^\text{tr}, \Zi)$. Let $\delta$ be the discriminant of $T$. Then $\delta P_s\in \text{End}_{\Zi-\text{HS}}(\mathbb{V}_s, \mathbb{V}_s)$ for all $s\in R$. Choose a polarization $q$ on $\text{End}_{\Zi-\text{HS}}(\mathbb{V}_s, \mathbb{V}_s)$ given by $q(A, B) = \text{tr}\, (A^TB)$. Then $q(\delta P_s, \delta P_s) = \text{tr}(\delta^2P_s^TP_s) = \text{tr}(\delta^2P_s^2) = \text{tr}(\delta^2P_s) = \delta^2\text{rk}\, T:= K$. By \cite[Theorem 1.1]{CDK95}, the subset $S^{(K)}\subseteq \text{HL}(\mathbb{V}^\otimes)$ of Hodge classes with norm bounded by $K$ is algebraic. Therefore, $S^{(K)} = S^\an$ because the subset $R\subseteq S^{(K)}$ is dense $S$. Hence, $\text{HL}(\mathbb{V}^\otimes) = S^\an$ that is an absurd.	This concludes the proof in the case when $B$ is an effective $\Q$-divisor.
	\end{proof}
	\begin{remark}
		 One can extend \cref{th: constructible} to the case when $B$ is an effective $\R$-divisor by applying Shokurov's polytope argument \cite[Proposition 3.2(3)]{Bir11}. 
	\end{remark}

	\subsection{General case}
	\label{sec: flag_schemes}
	In this subsection, let us remember that all schemes are assumed to be locally Noetherian and separated over an algebraically closed field $k$ of an arbitrary characteristic.
	\begin{definition}
		A \textit{flag scheme} $\underline{Z} = (Z_0; Z_{1}, Z_2 \dots , Z_n)$ consists of a proper scheme $Z_0$ and a finite sequence of closed subschemes $Z_i\subset Z_0$ such that $Z_0 \supset Z_{1} \supset Z_2 \supset \dots \supset Z_n$. The integer $l(\underline{Z}) = n \in \mathbb{Z}_{\ge 0}$ is called \textit{the length of} $\underline{Z}$. The scheme $Z_0$, denoted by $Z$, is called \textit{the ambient space of the flag scheme} $\underline{Z}$. A \textit{polarized flag scheme} $(\underline{Z}, A)$ consists of a flag scheme $\underline{Z}$ and an ample invertible sheaf $A$ on $Z$.
	\end{definition}
	Let us recall the notion of a scheme-theoretic image. Let $\varphi: Z\to Z^\prime$ be a proper morphism between Noetherian schemes. Then \textit{the scheme-theoretic image} $\text{Im}_\varphi(V)$ under $\varphi$ of a closed subscheme $\iota: V\subseteq Z$ is defined by the (coherent) sheaf of ideals $\mathcal{I} = \text{Ker}(\Os_{Z^\prime} \to (\varphi\circ\iota)_*\Os_V)$. If $V$ is reduced, then $\text{Im}_\varphi(V)$ is the set-theoretic image $\varphi(V)$ with the structure of a reduced closed subscheme of~$Z^\prime$.
	\begin{definition}
		We say two flag schemes $\underline{Z}$ and $\underline{Z^\prime}$ are \textit{isomorphic} if $l(\underline{Z}) = l(\underline{Z^\prime})$, and there exists an isomorphism $\varphi: Z\to Z^\prime$ such that $\text{Im}_\varphi(Z_i) = Z^\prime_i$ for each index $i = 0, 1, \dots, l(\underline{Z})$. Similarly, polarized flag schemes $(\underline{Z}, A)$ and $(\underline{Z^\prime}, A^\prime)$ are said to be isomorphic if there is an isomorphism $\varphi: \underline{Z}\to \underline{Z^\prime}$ such that $A \approx \varphi^*A^\prime$, the latter is an isomorphism of invertible sheaves.
	\end{definition}
	
	As in the case of varieties, we introduce families of flag schemes. Suppose $\underline{\Z}$ is a flag scheme, $f: {\Z} \to S$ is a proper flat morphism between Noetherian schemes. In addition, suppose that the restriction $f|_{\Z_i}: \Z_i \to S$ is flat as well for all $i = 1, 2, \dots l(\underline{\Z})$. Then the morphism $f$ is called \textit{a family of flag schemes}. It is denoted by $f: \underline{\Z} \to S$. If  the morphism $f$ is projective, then we shall say that $f$ is \textit{a projective family of flag schemes}. When there is no ambiguity, we denote $f$ by $\underline{\Z}/S$. As usual, the fiber of $\underline{\Z}/S$ over a point $s\in S$ is the flag scheme $((\Z_0)_s, (\Z_1)_s, \dots, (\Z_{l(\underline{\Z})})_s)$, and it is denoted by $\underline{\Z}_s$. As in \cref{subsec: CY} we put:
	\begin{equation*}
		\Iso_{\underline{Z}}(\underline{\Z}/S):= \{s\in S: \underline{\Z}_{\overline{s}} \approx \underline{Z}\times_k\overline{\kappa(s)}\}.
	\end{equation*}
	
	\begin{definition}
		\label{def: bounded_tuples}
		Let $\mathfrak{F}$ be a class of polarized flag schemes $(\underline{Z_\al}, A_\al)$. We say that the class $\mathfrak{F}$ is bounded if there exists a projective family of flag schemes $\underline{Z}/S$ with an ample invertible sheaf $\A$ on $\Z/S$ satisfying the following:
		\begin{enumerate}
			\item 
				For every $(\underline{Z_\al}, A_\al)$ in $\mathfrak{F}$ there is a closed point $s_\al\in S$ such that $(\underline{Z_\al}, A_\al) \approx (\underline{\Z}_{s_\al}, \A_{s_\al})$.
			\item 
				For every closed point $s\in S$ there exists $(\underline{Z_\al}, A_\al)$ in $\mathfrak{F}$ such that $(\underline{\Z}_{s}, \A_s)\approx (\underline{Z_\al}, A_\al)$.
		\end{enumerate}
	\end{definition}

	\begin{lemma}
		\label{lemma: immersion}
		Let $f:\Z\to S$ be a projective morphism, where $S=\Sp R$ is a Noetherian scheme. Let $\A$ be a very ample invertible sheaf on $\Z/S$ such that
		\begin{itemize}
			\item The natural map $H^0(\Z,\A)\otimes \kappa(s)\to H^0(\Z_s,\A_s)$ is an isomorphism for all $s\in S$.
			\item The direct image $f_*\A$ is a free sheaf of rank $N+1$.
		\end{itemize}
		Let $s_0,\dots,s_N$ be a basis of the free $R$-module $H^0(S,f_*\A)$, and
		$$\iota: \Z \hookrightarrow \Pp^N_S,\qquad x \longmapsto (s_0(x):\cdots:s_N(x))$$ be the corresponding closed immersion (over $S$). Then the ideal sheaf $\mathcal{I} = \textnormal{Ker}\left(\Os_{\Pp^N_S} \to \iota_*\Os_{\Z}\right)$ on $\Pp^N_S$ satisfies
		$
		H^0(\Pp^N_{\kappa(s)},\mathcal{I}_s(1))=H^1(\Pp^N_{\kappa(s)},\mathcal{I}_s(1))=0
		$
		for all points $s\in S$.
	\end{lemma} 
	\begin{proof}
		Recall that the given closed immersion (over $S$) $\Z\hookrightarrow \Pp^N_S$ corresponds to the surjective morphism $\mathcal{O}_\Z^{\oplus N+1}\to \A$, which is given by
		$
		(f_0,f_1,\dots,f_N)\longmapsto \sum_{i=0}^N f_i\cdot (s_i|_U)
		$
		on each non-empty open subset $U\subseteq \Z$. In particular, fiber-wise the immersions $\iota_s: \Z_s\hookrightarrow \Pp^N_{\kappa(s)}$ correspond to the induced surjections $\mathcal{O}_{\Z_s}^{\oplus N+1}\to \A_s$. We claim that the induced map
		$
		H^0(\Z_s,\mathcal{O}_{\Z_s}^{\oplus N+1})\longrightarrow H^0(\Z_s,\A_s)
		$
		is an isomorphism. Consider the natural commutative diagram:
		$$
		\xymatrix{
			{H^0(\Z,\mathcal{O}_\Z^{\oplus N+1})\otimes\kappa(s)} \ar@{->}[r] \ar@{->}[d]& {H^0(\Z,\A)\otimes\kappa(s)} \ar@{->}[d]\\
			{H^0(\Z_s,\mathcal{O}_{\Z_s}^{\oplus N+1})} \ar@{->}[r] & {H^0(\Z_s,\A_s).}
		}
		$$
		
		By the hypothesis, $H^0(\Z,\A)$ is a free $R$-module of rank $N+1$, hence the top horizontal arrow is an isomorphism. By construction, the left vertical arrow is an isomorphism. The right vertical arrow is an isomorphism by assumption. Therefore the bottom horizontal arrow is an isomorphism, as claimed.
		
		By construction, the closed immersion $\iota: \Z\hookrightarrow \Pp^N_S$ induces a closed immersion $\iota_s:\Z_s\hookrightarrow \Pp^N_{\kappa(s)}$  for every point $s\in S$ . By the claim, the restriction map
		\[
		H^0(\Pp^N_{\kappa(s)},\mathcal{O}_{\Pp^N_{\kappa(s)}}(1))\cong H^0(\Z_s,\mathcal{O}_{\Z_s}^{\oplus N+1})
		\longrightarrow
		H^0(\Z_s,\A_s)\cong H^0(\Pp^N_{\kappa(s)}, \iota_{s*}\mathcal{O}_{\Z_s}(1))
		\]
		is an isomorphism. Hence, from the short exact sequence of sheaves on $\Pp^N_{\kappa(s)}$
		\[
		0\to \mathcal{I}_s(1)\to \mathcal{O}_{\Pp^N_{\kappa(s)}}(1)\to \iota_{s*}\mathcal{O}_{\Z_s}(1)\to 0
		\]
		it follows that $H^0(\Pp^N_{\kappa(s)},\mathcal{I}_s(1))=H^1(\Pp^N_{\kappa(s)},\mathcal{I}_s(1))=0$.
	\end{proof}
	\begin{remark}
		If $\Z/S$ is a projective family of varieties, the conditions $$H^0(\Pp^N_{\kappa(s)},\mathcal{I}_s(1))=H^1(\Pp^N_{\kappa(s)},\mathcal{I}_s(1))=0$$ are equivalent to saying that the closed immersion (over $S$) $\iota_s: \Z\hookrightarrow \Pp^N_S$ is fiber-wise non-degenerate and given by the complete linear systems $|\A_s|$.
	\end{remark}
	
		 We say that a family of flag schemes $\underline{\Z}/S$ is \textit{\'etale locally trivial} if for every point $s_0\in S(k)$ there an \'etale neighborhood $(U, s_0)$ and a flag scheme $\underline{Z}$ such that $\underline{\Z}\times_S U \approx \underline{Z}\times_k U$ over~$U$.

	\begin{lemma}
		\label{lemma: glueing_isomorphisms}
		Let  $\underline{Z}$ be a flag scheme such that the ambient space $Z$ is a projective scheme,  $\underline{\Z}/S$ a projective family of flag schemes. Suppose that $\Iso_{\underline{Z}}(\underline{\Z}/S)(k) = S(k)$. Then $\Z_i/S$ is \'etale locally trivial for each $i = 0, 1, \dots, l(\underline{Z})$. Moreover, the family $\underline{\Z}/S$ is \'etale locally trivial if $\text{char}\, k = 0$.
	\end{lemma}
	\begin{proof}
		The first part of the lemma is a known result from deformation theory. In other words,
		every point $s_0\in S(k)$ has an \'etale neighborhood $(U, s_0)$ such that ${\Z_i}\times_S U \cong {Z_i}\times_k U$ over $U$ for each $i\in\{0, 1, \dots, l(\underline{Z})\}$ according to \cite[Proposition 2.6.10]{Ser06}.
		
		Next, we use the assumption $\text{char}\, k = 0$ to glue isomorphisms. For every (locally Noetherian) scheme $T$ over $S$ assign the set $${\mathcal{I}so}_S(T) = \{\varphi: \underline{Z}\times_k T \to \underline{\Z}\times_S T| \  \varphi \text{ is an isomorphism over } T \}.$$ The association $T\mapsto {\mathcal{I}so}(T)$ defies a contravariant functor from the category of locally Noetherian schemes over $S$ to the category of sets. It follows in the standard way (see \cite[Chapter 5, Theorem 5.22]{FGA05}) that the functor ${\mathcal{I}so}_S$ is representable by an open subscheme $\mathbf{Iso}_S/S$ of the flag Hilbert scheme $\mathbf{Hilb}_{(Z\times_k S)\times_S\Z/S}$ over $S$ (see \cite[Section 4.5]{Ser06}). Note that the structure morphism $I: \mathbf{Iso}_S \to S$ is surjective by the hypothesis $\Iso_{\underline{Z}}(\underline{\Z}/S)(k) = S(k)$. It follows from the generic smoothness on the source (in characteristic zero!) that the morphism $I$ admits a section \'etale locally, as required.
	\end{proof}
	\begin{remark}
		We expect that the assumption $\text{char}\, k = 0$ is unnecessary. One may adjust directly the argument in \cite[Proposition 2.6.10]{Ser06} for flag schemes.
	\end{remark}

	\begin{theorem}[Criterion of constructibility]
		\label{th: constructibility_criterion}
		Let  $\underline{Z}$ be a flag scheme such that the ambient space $Z$ is a projective integral scheme,  $\underline{\Z}/S$ a projective family of flag schemes. Assume either $\text{char}\, k = 0$ or $l(\underline{Z}) = 0$.  Then the following statements are equivalent:
		\begin{enumerate}
			\item
				The set $\Iso_{\underline{Z}}(\underline{\Z}/S)$ is constructible.
			\item 
				 The set $\{(\underline{\Z}_s, \A_s): s\in \textnormal{Iso}_{\underline{Z}}(\underline{\Z}/S)(k)\}$ is bounded for every ample invertible sheaf $\A$ on $\Z/S$.
			\item 
				The set $\{(\underline{\Z}_s, \A_s): s\in \textnormal{Iso}_{\underline{Z}}(\underline{\Z}/S)(k)\}$ is bounded for some ample invertible sheaf $\A$ on $\Z/S$.
			\item[($\star$)]
				There exists an ample invertible sheaf $\A$ on ${\Z}/S$, an integer $r\in \mathbb{\Zi}_{>0}$, and ample invertible sheaves $A_1, A_2, \dots, A_r$ on ${Z}$ such that for every point $s\in \Iso_{\underline{Z}}(\underline{\Z}/S)(k)$ there exists an isomorphism $\varphi_s: \underline{Z} \to \underline{\Z}_s$ with the property: $\varphi_s^*\A_s \equiv A_{i_s}$ for some $i_s\in \{1, 2, \dots r\}$.
		\end{enumerate} 
	\end{theorem}
	\begin{proof}
		The implication $(1)\implies (2)$ is immediate by \cref{def: bounded_tuples}, and $(2) \implies (3)$ is obvious. Let us prove $(3) \implies (\star)$. By \cref{def: bounded_tuples}, there is a family of polarized flag schemes $(\underline{\Z}^\prime/S^\prime, \A^\prime)$ such that for every closed point $s\in \Iso_{\underline{Z}}(\underline{\Z}/S)$ there is an isomorphism $\psi_ s: (\underline{\Z}^\prime_{s^\prime}, \A^\prime_{s^\prime}) \to (\underline{\Z}_s, \A_s)$ for some $s^\prime\in S^\prime(k)$, and vice versa (see \cref{def: bounded_tuples}). In particular, $\Iso_{\underline{Z}}(\underline{\Z^\prime}/S^\prime) = S^\prime$. By \cref{lemma: glueing_isomorphisms}, every point $s_0\in S^\prime(k)$ has an \'etale neighborhood $(U, s_0)$ such that $\underline{\Z^\prime}\times_{S^\prime} U \cong \underline{Z}\times_k U$ over $U$.
		For each point $s\in U(k)$ we denote its image in $S^\prime$ by the same symbol $s$ slightly abusing the notation.
		
		Let $\alpha:\underline{Z}\times_k U \to \underline{\Z^\prime}$ be the induced morphism. Hence, $\alpha_s: \underline{Z}\times_k s \to (\underline{\Z^\prime})_s$ is an isomorphism for every $s\in U(k)$ because the ground field $k$ is algebraically closed. By definition of the algebraic equivalence (see \cite{Ful98}), we obtain $(\alpha^*\A^\prime)_s \sim_\text{alg} (\alpha^*\A^\prime)_{s_0}$. Therefore, for every point $s\in U(k)$ there is an isomorphism $\varphi_s = \psi_s\circ \alpha_s: \underline{Z}\to \underline{\Z}_s$ such that $\varphi_s^*\A_s \sim_\text{alg} (\alpha^*\A^\prime)_{s_0}$. Then $\varphi_s^*\A_s \equiv (\alpha^*\A^\prime)_{s_0}$ Thus, invoking Noetherian induction on closed subsets of $S^\prime$, we prove the implication $(3)\implies (\star)$.
		
		Now, we prove the most delicate implication $(\star)\implies (1)$.
		Without loss of generality, we can assume that the base $S$ is irreducible. By Noetherian induction on closed subsets of $S$, it is sufficient to prove the statement over an open dense subset. Hence, we may assume that the base $S = \text{Spec}\, R$ is an affine integral scheme (over $k$). Let $f: \Z\to S$ be the given morphism, $n = l(\underline{Z})$.\\
		\textbf{Step 1.} We reduce the problem to proving the constructibility of a subset of the flag Hilbert scheme  that parametrize the closed flag subschemes (of a projective space) isomorphic to $\underline{Z}$. Choose a sufficiently large integer $M\in \Zi_{>0}$ such that $\A^{\otimes M}$ is very ample on $\Z/S$ and $H^1(\Z_\eta, \A_\eta^{\otimes M}) = 0$ (Serre's vanishing theorem), where $\eta \in S$ is the generic point. By the Cohomology and Base Change Theorem \cite[Chapter III, Theorem 12.11]{Har77}, after shrinking the base, for every $s$ in $S$ we have:
		\begin{enumerate}
			\item [(i)]
				$H^1(\Z_s, \A_s) = 0$.
			\item [(ii)]
				The natural map $H^0(\Z, \A)\otimes \kappa(s) \to H^0(\Z_s, \A_s)$ is an isomorphism.
			\item [(iii)]
				The direct image $f_*\A$ is a free sheaf of some rank $N+1\ge1$.
		\end{enumerate}
		According to \cref{lemma: immersion}, there is a closed immersion $\iota: \Z\hookrightarrow \Pp^N_S$ over $S$ such that the following holds for every point $s\in S$: $H^0(\Pp^N_{\kappa(s)},\mathcal{I}_s(1))=H^1(\Pp^N_{\kappa(s)},\mathcal{I}_s(1))=0$. Since $\Z_j \subset \Z$ for each $j = 1, \dots n$, we shall write $\iota: \underline{\Z} \hookrightarrow \Pp^N_S$.
		Let $$\mathbf{P}(t): = \mathbf{P}_\eta(t) = \left(\chi(\Z_\eta, \A_\eta^{\otimes Mt}); \chi((\Z_1)_\eta, \A_\eta^{\otimes Mt}), \dots, \chi((\Z_n)_\eta, \A_\eta^{\otimes Mt})\right)\in \Q[t]^{\times (n+1)}$$ be the Hilbert polynomial of $\underline{\Z}_\eta$ with respect to $\A_\eta^{\otimes M}$. For every $s\in S$ the Hilbert polynomial $\mathbf{P}_s(t)$ of $\underline{\Z}_s$ with respect to $\A^{\otimes M}_s$ is defined similarly. By the flatness of $f$ and the connectedness of $S$ we conclude that $\mathbf{P}_s(t) = \mathbf{P}(t)$ for every $s\in S$.
		
		Next, by $\Hilb$ we denote the flag Hilbert scheme of $\Pp^N_k$ relative to $\mathbf{P}(t)$ \cite[Theorem 4.5.1]{Ser06}. By universality of $\Hilb$, there is a natural map $h: S\to \Hilb$ which sends every point $s\in S(k)$ to $[\iota_s: \underline{\Z}_s\hookrightarrow \Pp^N_k]$. Let us define the following subset:
		\begin{equation*}
			\text{LS}_Z = \begin{Bmatrix}
				 & \underline{V}\approx \underline{Z} \\ [\iota_V: \underline{V}\hookrightarrow \Pp^N_k]: &  H^1\left(V, \Os_{\Pp^N_k}(1)|_V\right) = 0\\ & H^0\left(\Pp^N_k, \mathcal{I}_V(1)\right) =  H^1\left(\Pp^N_k, \mathcal{I}_V(1)\right) = 0
			\end{Bmatrix} \subseteq \Hilb (k),
		\end{equation*}
		where $\mathcal{I}_V = \text{Ker}\left(\Os_{\Pp^N_k} \to (\iota_V)_*\Os_V\right)$.
		We claim that the subset $\LS\subseteq \Hilb(k)$ is constructible. Suppose that the claim holds. Recall that the base field $k$ is algebraically closed. Then we have $h^{-1}(\LS) \subseteq \Iso_{\underline{Z}}(\Z/S) (k)$ by the definition of the flag Hilbert scheme. The reverse inclusion holds by the construction of the immersion $\iota: \underline{Z} \hookrightarrow \Pp^N_S$. Hence the set $\Iso_{\underline{Z}}(\underline{\Z}/S)(k) = h^{-1}(\LS)$ is constructible in $S(k)$. Furthermore, the family $\underline{\Z}/S$ is \'etale locally trivial due to \cref{lemma: glueing_isomorphisms}. Thus, the subset  $\Iso_{\underline{Z}}(\underline{\Z}/S) \subseteq S$ is constructible. Now let us prove the claim.
		
		\textbf{Step 2.} We show that certain locally closed subschemes of the Picard scheme $\Pic$ determine the desired subset $\LS$. By \cite[Chapter 9, Theorem 9.4.8]{FGA05}, the Picard scheme $\Pic$ exists because $Z$ is an (geometrically) integral projective scheme. Moreover, there exists a universal Poincar\'e sheaf $\Pu$ on $Z\times_k \Pic$ since $Z(k)\neq \emptyset$ (see \cite[Chapter 9, Exercise 9.4.2]{FGA05}). Choose an ample invertible sheaf $A$ on $Z$. Let us recall that the Hilbert polynomial $P(t)$ of a coherent sheaf $L$ on a projective scheme $Z$ with respect to $A$ is the polynomial $P_{A}[L](t) = \chi (Z, L\otimes A^{\otimes t}) \in \Q[t]$. Given a polynomial $P(t)\in \Q[t]$, let $\Pic^{P(t)}\subseteq \Pic$ be the set of points representing invertible sheaves $L$ such that $P_A[L](t) = P(t)$. By \cite[Chapter 9, Theorem 9.6.20]{FGA05}, the schemes $\Pic^{P(t)}$ are open and closed subschemes of $\Pic$; they are disjoint and cover. Moreover, all $\Pic^{P(t)}$ are quasiprojective. It follows from the hypothesis ($\star$) that for every closed point $s\in \Iso_{\underline{Z}}(\underline{\Z}/S)$ there exists an isomorphism $\varphi_s: \underline{Z}\to \underline{\Z}_s$ such that $P_A[\varphi_s^*\A^{\otimes M}](t) = P_A[A_{i_s}^{\otimes M}]$ for some $i_s\in \{1, 2, \dots, r\}$. Therefore, we obtain the inclusion $$\begin{Bmatrix}
			[\varphi^*_s\A_s^{\otimes M}]\in \Pic: s\in \Iso_{\underline{Z}}(\underline{\Z}/S)(k)
		\end{Bmatrix} \subseteq \bigcup_{i=1}^r \Pic^{P_A[A_i](t)}.$$ Without loss of generality, we may assume that $r = 1$. Let $P(t) = P_A[A_1^{\otimes M}](t)$.
		
		\textbf{Step 2.1.} We show that the following subset is constructible $$\VA = \{[L]\in \Pic^{P(t)}: L \text{ is very ample}, h^0(Z, L) = N+1, h^1(Z, L) = 0, \ \chi(\underline{Z}, L^{\otimes t}) = \mathbf{P}(t)\}.$$ Indeed, by the general result of Grothendieck (see \cite[Appendix E]{GW20}) the set $\{s\in \Pic^{P(t)}: \Pu_s \text{ is very ample}\}$ is constructible. By the semicontinuity theorem \cite[Chapter III, Theorem 12.8]{Har77}, the set $\{s\in \Pic^{P(t)}: h^1(Z, \Pu_s) = 0\}$ is open. Similarly, the set $\{s\in \Pic^{P(t)}: h^0(Z, \Pu_s) = N + 1\}$ is constructible. In addition, the Euler characteristic $\chi(Z_i, \Pu_s)$ is locally constant in families for each $i = 0, 1, \dots, n$. Then the set $\VA$ is constructible. Hence after stratification of $\VA$, we may assume that $\Pu|_{\VA}$ is a very ample invertible sheaf on $Z\times_k \VA / \VA$. Without further notice, we denote the pullbacks of $\Pu$ under a base change by the same symbol $\Pu$ when no confusion arises.
		
		\textbf{Step 2.2.} We construct an open cover $\bigcup_i U_i$ of $\VA$ together with morphisms $h_i: U_i \to \Hilb$ such that $\cup_i h(U_i)(k)$ equals to $\LS$ up to the natural action of $\Aut(\Pp^N_k)$. By construction, we have $H^1(Z, \Pu_s) = 0$ for every $s\in \VA$. This yields that there is a finite affine cover $\VA= \bigcup_i U_i$ such that the conditions (i)-(iii) from Step 1  hold for $\Pu|_{U_i}$. By \cref{lemma: immersion}, there are closed immersions $\iota_i: \underline{Z} \times_k U_i \hookrightarrow \Pp^N_{U_i}$ such that $H^0(\Pp^N_k, (\mathcal{I}_i)_s(1)) = H^1(\Pp^N_{k}, (\mathcal{I}_i(1))_s) = 0$ for all $s\in U_i(k)$ and all $i$. Here, $\mathcal{I}_i = \text{Ker}\left(\Os_{\Pp^N_{U_i}}\to (\iota_i)_*\Os_Z\right)$. Therefore, there exist morphisms $h_i: U_i\to \Hilb$. By construction, $\VA^h: = \bigcup_i h_i(U_i)(k)\subseteq \LS$. By Chevalley's Theorem \cite[Theorem 1.8.4]{EGAIV}, the set $\VA^h$ is constructible.
		
		\textbf{Step 2.3.} We prove that $\Aut(\Pp^N_k)\cdot \VA^h = \LS$. For every point $[g: \underline{Z}\hookrightarrow \Pp^N_k]\in \LS$ we can canonically assign the point $[g^*\Os(1)]\in \VA$. However, the immersion $\iota: \underline{Z}\hookrightarrow\Pp^N_k$ constructed above may differ from $g$ by an automorphism of the projective space. Let $\PGL$ be the quasiprojective group scheme with $\PGL(k) = \Aut(\Pp^N_k)$. There is a natural (left) action $F: \PGL\times_k \Hilb \to \Hilb$ defined on closed points by $F(\varphi, [g]) = [\varphi \circ g]$. Since $F(\VA^h) = \LS$, the set $\LS$ is constructible by Chevalley's Theorem. This completes the proof.
	\end{proof}
	
	\begin{corollary}
		\label{cor: criterion_pairs}
		Assume $\text{char}\, k = 0$. Let  $(X, B)$ be a projective pair,  $(\X/S, \B)$ a projective elementary family. Then the following statements are equivalent:
		\begin{itemize}
			\item
			The set $\Iso_{(X, B)}(\X/S, \B)$ is constructible.
			\item
			There exists an ample invertible sheaf $\A$ on ${\X}/S$, an integer $r\in \mathbb{\Zi}_{>0}$, and ample line bundles $A_1, A_2, \dots, A_r$ on ${X}$ such that for every point $s\in \Iso_{(X, B)}(\X/S, \B)(k)$ there exists an isomorphism $\varphi_s: {(X, B)} \to (\X_s, \B_s)$ with the property: $\varphi_s^*\A_s \equiv A_{i_s}$ for some $i_s\in \{1, 2, \dots r\}$.
		\end{itemize} 
	\end{corollary}
	\begin{proof}
		Let $\underline{Z} = \underline{Z}^{(X, B)}$ (resp. $\underline{\Z} = \underline{\Z}^{(\X, \B)}$) be the flag scheme associated with $(X, B)$ (resp. $(\X, \B)$) as in \cref{const: associated_flag}. By the generic flatness theorem, \cite[Theorem 5.12]{FGA05} we may assume that $\underline{\Z}$ is a family of flag schemes. According to \cref{remark: iso_pair_and_flag}, we have $\Iso_{(X, B)}(\X/S, \B) = \Iso_{\underline{Z}}(\underline{\Z}/S)$ after shrinking the base. Then the theorem follows from \cref{th: constructibility_criterion} by Noetherian induction on closed subsets of $S$.
	\end{proof}

	Recall that for an integral scheme $Z$ (over $k$) the group of Cartier divisors modulo the linear equivalence is naturally isomorphic to the group $\textnormal{Pic}(Z)$ of invertible sheaves on $Z$ modulo isomorphisms. Suppose $Z$ is projective. By $\text{NS}(Z) = \text{Pic}(Z)/\equiv$ denote the group of Cartier divisors modulo the numerical equivalence, let $N^1_\Q(Z) = \text{NS}(Z)\otimes_\Zi \Q, N^1_\R(Z) = N^1_\Q(Z)\otimes_\Q \R$. Let $\text{Amp}(Z) \subseteq N^1_\R(Z)$ (resp. $\text{Nef}(Z) \subseteq N^1_\R(Z)$) denote the convex cone generated by all ample Cartier divisors (resp. nef $\R$-Cartier $\R$-divisors). According to \cite[Chapter IV, Theorem 1]{Kle66}, we have $\text{Nef}(Z) = \overline{\text{Amp}(Z)}$, and $\text{Amp}(Z) = \text{Int} \left(\text{Nef}(Z)\right)$. Next, let $\text{Eff}(Z) \subseteq N^1_\R(Z)$ be the convex cone generated by all effective Cartier divisors. Put $\text{Nef}^e(Z) = \text{Nef}(Z)\cap \text{Eff}(Z)$. Now, suppose $\underline{Z}$ is a flag scheme. Then we define its automorphism group $\text{Aut}(\underline{Z}) = \{\varphi\in \text{Aut}(Z): \ \text{Im}_\varphi (Z_i) = Z_i \ \forall i = 0, 1, \dots, l(\underline{Z})\}$, which naturally acts on $\text{Nef}^e(Z)$. For a pair $(X, B)$ we set $\Aut(X, B) = \Aut(\underline{Z}^{(X,B)})$, where $\underline{Z}^{(X, B)}$ is the flag scheme associated with $(X, B)$ (see \cref{const: associated_flag}).

	We say that  $\textnormal{Nef}^e(Z)$ \textit{admits a rational polyhedral fundamental domain for the natural action of} $\textnormal{Aut}(\underline{Z})$ if there is a rational polyhedral cone $\Pi\subseteq \textnormal{Nef}^e(Z)$:
	\begin{equation*}
		\text{Aut}(\underline{Z})\cdot \Pi = \text{Nef}^e(Z),
	\end{equation*}
	and for every $\varphi\in \text{Aut}(\underline{Z})$ we have $\text{Int}(\varphi^*\Pi)\cap \text{Int}(\Pi) = \emptyset$ unless $\varphi^* = \text{id}$. We refer to \cref{Appendix Cone_Conjecture} for further discussions.
	\begin{corollary}
		\label{cor: cone_and_constructibility}
		Let  $\underline{Z}$ be a flag scheme such that the ambient space $Z$ is a projective integral scheme,  $\underline{\Z}/S$ a projective family of flag schemes. Assume either $\text{char}\, k = 0$ or $l(\underline{Z}) = 0$. Then the subset $\Iso_{\underline{Z}}(\underline{\Z}/S)\subseteq S$ is constructible  if $\textnormal{Nef}^e(Z)$ admits a rational polyhedral fundamental domain for the natural action of $\textnormal{Aut}(\underline{Z})$.
	\end{corollary}
	\begin{proof}
		Choose an ample invertible sheaf $\A$ on $\Z/S$, and set $N = \A_{s_0}^{d}>0$ for some $s_0\in \Iso_{\underline{Z}}(\underline{\Z}/S)(k)$, where $d = \dim Z$. Then for every $s\in \Iso_{\underline{Z}}(\underline{\Z}/S)(k)$ there is an isomorphism $\varphi: \underline{Z}\to \underline{\Z}_s$ such that $$[\varphi_s^*\A_s] \in \{[A]\in \text{NS}(Z): A^{d} = N, \ A \text{ is ample}\}.$$
		We claim that the number of orbits for the natural action of $\Aut(\underline{Z})$ on the defined set is finite. Indeed, there is a rational polyhedral fundamental domain $\Pi\subseteq \textnormal{Nef}^e(Z)$ for $\Aut(\underline{Z})$, by the hypothesis. Let $x_1, x_2, \dots, x_d\in \Pi$ be linearly independent elements, then $x_{i_1}\cdot x_{i_2}\cdot \dots \cdot x_{i_d}\ge 0$ for every tuple $1\le i_1 \le i_2 \le \dots \le i_d\le d$, and $x_{i_1}\cdot x_{i_2}\cdot \dots \cdot x_{i_d}> 0$ if all indices are different. Since $\Pi\cap N^1_\Zi(X/Y)$ is a finitely generated semigroup, it follows that the set $\{[A]\in \Pi\cap N^1_\Zi(X/Y): A^d = N\}$ is finite. Now, the statement follows from \cref{th: constructibility_criterion}.
	\end{proof}

	\section{Applications}
	\label{sec: main}
	\subsection{Constructibility} Combining \cref{cor: criterion_pairs} and \cref{th: constructible} we obtain the following.
	\begin{theorem}
		\label{th: Fano_0_General}
		Let $(\X/S, \B)$ be a projective elementary family, $(X, B)$ be a projective klt log pair, and $B$ is an effective $\Q$-divisor. Then the subset $\Iso_{(X, B)}(\X/S, \B) \subseteq S$ is constructible in each of the following cases:
		\begin{enumerate}
			\item[($-$)] 
				Fano: $-(K_X + B)$ is ample.
			\item[($0$)]
				Calabi--Yau: $K_X + B\equiv 0$.
			\item[($+$)] 
				General type: $K_X + B$ is ample.
			\item[($\star$)]
				$\Nef^e(X)$ admits a rational polyhedral fundamental domain for $\Aut(X, B)$. 
		\end{enumerate}
		Moreover, in case $(0)$ the subsets $\Iso_{(X, B)}^1(\X/S, \B), \Wlc_{(X, B)}(\X/S, \B) \subseteq S$ are also constructible.
	\end{theorem}
	\begin{proof}
		Indeed, case $(0)$ is content of \cref{th: constructible}. In addition, case $(\star)$ is rephrasing of \cref{cor: cone_and_constructibility} in terms of pairs. Let us prove cases $(\pm)$ arguing by Noetherian induction on closed subsets of $S$. By \cref{lemma: klt_0_families}, we may assume that $(\X, \B)$ is a klt log pair. Set $\A = \Os_\X(m(K_\X + \B))$ for $m\gg 1$. According to \cref{reamrk: restriction_canonical}, we have $\A_s = \Os_{\X_s}(m(K_{\X_s} + \B_s))$ for all points $s\in S$. Then for every $s\in \Iso_{(X, B)}(\X/S, \B)(\Ko)$ there is a log isomorphism $\varphi_s: (X, B) \to (\X_s, \B_s)$ such that $\varphi_s^* \A_s \equiv \Os_X(m(K_X + B))$. Since ampleness is an open condition in families \cite[Theorem 1.2.17]{Laz04}, either $\A$ or its dual $-\A$ is ample on $\X/S$ after shrinking the base. This implies that the subset $\Iso_{(X, B)}(\X/S, \B)$ is constructible due to \cref{cor: criterion_pairs}.
	\end{proof}
	\begin{remark}
		\cref{th: Fano_0_General} can be generalized to the case when either $-(K_X +B)$ or $K_X + B$ is big and nef by applying standard techniques of the Minimal Model Program \cite{BCHM10}. 
	\end{remark}
	\begin{remark}
		Suppose $K_X + B$ is big and nef, then the number of wlc models for $(X, B)$ is finite up to natural isomorphism. Hence, constructibility of $\Iso_{(X, B)}(\X/S, \B) \subseteq S$ immediately implies constructibility of the subsets $\Iso_{(X, B)}^1(\X/S, \B), \Wlc_{(X, B)}(\X/S, \B) \subseteq S$.
	\end{remark}
	
	\begin{conjecture}
		Let $(\X/S, \B)$ be a projective elementary family, $(X, B)$ be a projective wlc model. Then the subsets $\Iso_{(X, B)}(\X/S, \B), \Iso_{(X, B)}^1(\X/S, \B), \Wlc_{(X, B)}(\X/S, \B)\subseteq S$ are constructible.
	\end{conjecture}
	\subsection{Consequences of the cone conjecture}
	Let $\mathfrak{C}$ be a class of pairs. By $|\Cc|$ we denote the set of (log) isomorphism classes included in $\Cc$.
	\begin{theorem}
		\label{th: main}
		Let $\Cc$ be a class of $0$-equivalent projective klt $0$-pairs with $\Q$-coefficients. Suppose that the class $\Cc$ admits a bounded polarization. Then the set $|\Cc|$ is finite.
	\end{theorem}
	\begin{proof}
		According to \cref{prop: bounded_subclass}, there exists finitely many projective elementary families $(\X^{(j)}/S^{(j)}, \B^{(j)})$ such that for every pair $(X_\alpha, B_\alpha)$ in $\Cc$ there is an index $j$ together with a closed point $s_\alpha\in S^{(j)}$, and an isomorphism $\varphi_\al: (X_\al, B_\al) \to (\X_{s_\al}, \B_{s_\al})$. Without loss of generality, we may assume there is a single family $(\X/S, \B)$, and the set $\{s_\al: \alpha\in I\}$ is dense in $S$. By Noetherian induction on closed subsets $S^\prime\subseteq S$ such that $\{s_\al: \al\in I\}\cap S^\prime$ is dense in $S^\prime$, it is sufficient to show that the family $(\X/S, \B)$ is isotrivial over an open dense subset. Due to \cref{th: Fano_0_General}, we can assume that $S = \Wlc_{(X, B)}(\X/S, \B)$.
		
		Next, we show that boundedness of $\Cc$ together with \cref{th: Fano_0_General} implies finiteness of $|\Cc|$. For every $\Cc_i\in |\Cc|$ choose a representative $(X_{\al_i}, B_{\al_i})$ in $\Cc_i$. Then $S = \bigsqcup_i \Iso_{(X_{\al_i}, B_{\al_i})}(\X/S, \B)$. From normality of the varieties $X_\al$ it follows that the set $|\Cc|$ is at most countable. Hence, one of the constructible sets $\Iso_{(X_{\al_i}, B_{\al_i})}(\X/S, \B)$ includes a dense open subset as the base field $\Ko$ is uncountable. This concludes the proof.
	\end{proof}
	
	 Let $(X, B)$ be a proper pair. Recall that $\text{NS}(X) = \mathrm{Pic}(X)\mathbin{/}\!\equiv$ denotes the group of invertible sheaves on $X$ modulo the numerical equivalence, and $\text{Aut}(X, B) = \{\varphi\in \text{Aut}(X): \varphi_*(B) = B\}$.
	\begin{theorem}
		\label{th: finiteness of orbits}
		Fix $N\in \mathbb{Z}_{>0}$. Let $(X, B)$ be a projective klt $0$-pair, $B$ is an effective $\Q$-divisor. Then the number of orbits for the natural action of $\textnormal{Aut}(X, B)$ on the set $$\{[A]\in \textnormal{NS}(X): A^{\dim X} = N, \ A \text{ is ample}\}$$ is finite (whenever the defined set is non-empty).
	\end{theorem}
	\begin{proof}
		Set $d = \dim X$. We use Big Matsusaka's type theorem to reduce the problem to a set of very ample invertible sheaves on $X$. Let $\text{Amp}_{\mathbb{Z}}(X) \subseteq \text{NS}(X)$ be the subset of all classes of ample invertible sheaves.  According to the Effective Base Point freeness \cite[Theorem 1.1]{Kol93}, there exists $m'=m'(d)\in \Zi_{>0}$ such that $A^{\otimes m'}$ is base point free for every $[A]\in \text{Amp}_{\Zi}(X)$. Furthermore, there exists $m=m(d)\in \Zi_{>0}$ such that $A^{\otimes m}$ is very ample for every $[A]\in \text{Amp}_{\Zi}(X)$ (cf. \cite[Lemma 1.2]{Kol93} and \cite[Theorem]{Wil80}). Let $\{A_\al: \al \in I\}$ be a set representing all classes in $$\{[A]\in \textnormal{Amp}_{\mathbb{Z}}(X): A^{d} = N\}.$$
		
		Next, we show that the intersection numbers $B\cdot (A_\al)^{d-1}$ are bounded. By the standard Hilbert-Chow schemes argument, there exist finitely many polynomials $P_i(t)\in \Q[t]$ such that for every $\alpha\in I$ we have $\chi(X, A_\al^{\otimes mt}) = P_{i_\al}(t)$ for some $i_\al$. Without loss of generality, we suppose that there is a single polynomial $P(t)$ such that $\chi(X, A_\al^{\otimes mt}) = P(t)$ for all $\al\in I$. Let $g: Y \to X$ be a resolution of singularities. Applying the Hirzebruch--Riemann--Roch formula \cite[Theorem 4.1]{Har77} to $g^*A_\al^{\otimes mt}$ on the smooth projective variety $Y$  we obtain:
		\begin{equation*}
			\chi(Y, g^*A_\al^{\otimes mt}) = \frac{(g^*A_\al)^d}{d!}(mt)^d - \frac{K_Y\cdot (g^*A_\al)^{d-1}}{2(d-1)!}(mt)^{d-1} + O(t^{d-2}).
		\end{equation*}
		Note that $\chi(Y, g^*A_\al^{\otimes mt})=\chi(X, A_\al^{\otimes mt})$ because klt singularities are rational \cite[Theorem 5.22]{KM98}. Moreover, it follows from the projection and restriction formulas \cite{Ful98} that
		\begin{equation*}
			P(t) = \frac{(A_\al)^d}{d!}(mt)^d - \frac{K_X\cdot (A_\al)^{d-1}}{2(d-1)!}(mt)^{d-1} + O(t^{d-2}).
		\end{equation*}
		Put $N = 2P^{(d-1)}(0)/m^{d-1}$. Hence the intersection numbers $B\cdot(A_\alpha)^{d-1}=-K_{X_\alpha}\cdot A_\al^{d-1} = N$ are the same for all $\al\in I$.
		
		It remains to bound $(B)_{\mathrm{red}}\cdot A_\alpha$. Set $b=\min\{b_i: b_i\neq 0\}$, and write $B=\sum_{i=1}^l b_{i}B_{i}$. To be definite, we can assume that $b_{i}\neq 0$ for all $i$.  Then
		$$
		b\,B_{i}\cdot (A_\al)^{d-1}\ \le\ \sum_{i} b\,B_{i}\cdot (A_\alpha)^{d-1}\ \le\ \sum_i b_{i}\, B_{i}\cdot (A_\al)^{d-1}\ = B \cdot (A_\al)^{d-1} \le\ N.
		$$
		Therefore, $B_{\mathrm{red}}\cdot (A_\alpha)^{d-1} \le l\,\dfrac{N}{b}$ because $\#\supp (B) = l$.
		
		Now, we use constructibility to conclude that the set of triples $(X, B; A_\al)$ is bounded. More precisely, according to \cref{prop: bounded_subclass} there exists finitely many projective elementary families $(\X^{(j)}/S^{(j)}, \B^{(j)})$ together with very ample invertible sheaves $\A^{(j)}$ on $\X^{(j)}/S^{(j)}$ such that for every $\alpha \in I$ there is an index $j$ together with a closed point $s_\alpha\in S^{(j)}$, and an isomorphism $\varphi_{\al}\colon (X, B; A_{\al}^{\otimes m})\to (\X_{s_\al}, \B_{s_\al}; \A^{(j)}_{s_\al})$.  Without loss of generality, we may assume there is a single triple $(\X/S, \B; \A)$. Then the subset $\Iso_{(X, B)}(\X/S, \B)\subseteq S$ is constructible due to \cref{th: Fano_0_General}. Therefore, the theorem follows from the criterion of constructibility, i.e. \cref{th: constructibility_criterion}.
	\end{proof}
	
	Next, we prove a natural analogue of \cref{th: finiteness of orbits}. Let $X$ be a proper normal variety of dimension $d\in \Zi_{>0}$. For an invertible sheaf $M$ on $X$ define its \textit{volume}:
	\begin{equation*}
		\text{vol}_X(M) = \underset{m\to +\infty}{\lim\sup}\frac{h^0(X, M^{\otimes m})}{m^d/d!}.
	\end{equation*}
	 By definition $M$ is \textit{big} if $\text{vol}_X(M)>0$. Note that if $M$ is nef, then $\text{vol}_X(M) = M^d$ \cite[Corollary 1.4.41]{Laz04}.  We shall say $M$ is \textit{movable} if $H^0(X, M)\neq 0$, and the indeterminacy locus of the natural map\footnote{Closed points of $\text{Proj}\,\left(\text{Sym}\, H^0(X, M)\right)$ are naturally identified with hyperplanes of $H^0(X, M)$.} $\psi_M: X \dashrightarrow \text{Proj}\,\left(\text{Sym}\, H^0(X, M)\right)$ has codimension $\ge 2$. A movable invertible sheaf $M$ is \textit{very movable} if $\psi_M$ is small onto its image. Let $(X,B)$ be a pair, put $$\text{PsAut}(X, B) = \{\varphi\colon X \overset{\textnormal{small}}{\longdashleftrightarrow} X|\, \varphi_*(B) = B\}.$$ When $X$ is $\Q$-factorial, and $(X, B)$ is a trm $0$-pair we have $\text{PsAut}(X, B) = \text{Bir}(X, B)$ by \cref{prop: flops_0-equivalent}.
	\begin{theorem}
		\label{th: finiteness of orbits(movable)}
		Fix $N\in \mathbb{Z}_{>0}$. Let $(X, B)$ be a projective klt $0$-pair, $B$ is an effective $\Q$-divisor. Then the number of orbits for the natural action of $\textnormal{PsAut}(X, B)$ on the set $$\{[M]\in \textnormal{NS}(X): \textnormal{vol}_X(M) = N, \ M \text{ is big and very movable }\}$$ is finite (whenever the defined set is non-empty).
	\end{theorem}
	\begin{proof}
		Let $\{M_\al: \al \in I\}$ be a set of effective Cartier divisors representing all classes in $$\{[M]\in \textnormal{NS}(X): \textnormal{vol}_X(M) = N, \ M \text{ is big and very movable }\}.$$
		Note that for every $\alpha \in I$ there is a rational number $0<\epsilon_\al \ll 1$ such that $(X, B + \epsilon_\al M_\al)$ is klt, and $K_X + B +\epsilon_\al M_\al$ is big. Slightly abusing the notation, we do not make distinction between Cartier divisors and invertible sheaves. Then the canonical ring $R(X, M_\al) = \bigoplus_{m\ge 0}H^0(X, M_\al)$ is finitely generated due to \cite[Theorem 3.11]{KM98}. Set $\psi_\al : X \dashrightarrow X_\al  = \text{Proj}\,R(X, M_\al), B_\al = (\psi_\al)_* B, A_\al = (\psi_\al)_* M_\al$, and $d = \dim X$. Since $M_\al$ is big and very movable, the map $\psi_\al$ is small for all $\al\in I$, and $A_\al$ is ample. In accordance with \cref{prop: flops_0-equivalent}, $(X_\al, B_\al)$ is a projective klt $0$-pair in the $0$-class of $(X, B)$. By \cite[Lemma 2.2.1]{HMX18}, $A_\al^d = \text{vol}_{X_\al}(A_\al) \le \text{vol}_X(M_\al) = N$. As in the proof of \cref{th: finiteness of orbits}, there exists an integer $m\in \Zi_{>0}$ such that $A_\al^{\otimes m}$ is very ample for all $\al\in I$, and the intersection numbers $(B_\al)_\text{red}\cdot (A_\al^{\otimes m})^{d-1}$ are bounded.
		
		By \cref{prop: bounded_subclass}, there exists finitely many projective elementary families $(\X^{(j)}/S^{(j)}, \B^{(j)})$ together with very ample invertible sheaves $\A^{(j)}$ on $\X^{(j)}/S^{(j)}$ such that for every $\alpha \in I$ there is an index $j$ together with a closed point $s_\alpha\in S^{(j)}$, and an isomorphism $\varphi_{\al}\colon (X_\al, B_\al; A_{\al}^{\otimes m})\to (\X_{s_\al}, \B_{s_\al}; \A^{(j)}_{s_\al})$.  Without loss of generality, we may assume there is a single triple $(\X/S, \B; \A)$. Then the subset $\Iso^1_{(X, B)}(\X/S, \B)\subseteq S$ is constructible due to \cref{th: Fano_0_General}. Then we may assume $\Iso^1_{(X, B)}(\X/S, \B) = S$. Hence, the base $S = \bigsqcup_{\al \in I}\Iso_{(X_\al, B_\al)}(\X/S, \B)$ is an (at most) countable union of constructible sets by \cref{th: Fano_0_General}. Since the base field $\Ko$ is uncountable, there is an index $\al_0\in I$ such that $\Iso_{(X_{\al_0}, B_{\al_0})}(\X/S, \B)$ includes an open dense subset $U\subseteq S$. Hence, we may assume $S = \Iso_{(X_{\al_0}, B_{\al_0})}(\X/S, \B)$. Furthermore, we can suppose that $(\X, \B) = (X_{\al_0}, B_{\al_0})\times_\Ko S$ after an \'etale base change. Therefore, we obtain numerical equivalence $\A_{s}\equiv \A_{s_{\al_0}}$ for all closed $s\in S$. Then for each $\al\in I$ we consider the map $\Psi_\al = (\psi_{\al_{0}})^{-1}\circ\varphi_\al\circ\psi_\al: (X,B)\dashrightarrow (X, B)$. By construction, $\Psi_\al \in \text{PsAut}(X, B)$ and $\Psi_\al M_\al^{\otimes m} \equiv \Psi_{\al_0}M_{\al_0}^{\otimes m}$ for all $\al \in I$. This completes the proof.
	\end{proof}
	\begin{remark}
		When $X$ is $\Q$-factorial, it is interesting whether one may extend \cref{th: finiteness of orbits} (resp.  \cref{th: finiteness of orbits(movable)}) to all big and nef (resp. big and movable) invertible sheaves of fixed volume, by the very technique used in the proofs.
	\end{remark}
	
	\subsection{Counterexamples}
	\label{sec: counterexamples}
	Suppose $\X/S$ is a projective family of varieties, and $X$ is a smooth projective variety. First, we show that one can not expect that the set $\Iso_X(\X/S)$ is constructible if $X$ is neither a Mori model or a minimal model. The following is based on \cite[Section 1.6]{Huy18}.
	\begin{example}
		\label{ex: blowup_k3}
		Let $S$ be a smooth projective K3 surface such that $\text{Aut}(X)$ is infinite and countable. In \cite{Weh88}, one can find examples of projective K3 surfaces with Picard number 2 such that the automorphism group is isomorphic to $\Zi/2 * \Zi/2$ (free product). Consider the trivial family $p_2: S\times_\Ko S \to S$. Let $\X = \text{Bl}_\Delta(S\times_{\Ko} S)$ be the blow-up of the diagonal $\Delta \subset S\times_\Ko S$, $\sigma: \X \to S\times_\Ko S$ the corresponding morphism. Put $f =p_2\circ\sigma: \X \to S$. Choose a (very general) closed point $s_0\in S$ such that the set $\{\varphi(s_0): \varphi\in \text{Aut}(S)\}$ is infinite. Let $X = \X_{s_0}$. Since $S$ is a minimal surface, $s\in \Iso_X(\X/S)$ if and only if there is an automorphism $\varphi\in \text{Aut}(S)$ such that $\varphi(s_0) = s$. Hence, $\text{Iso}_X(\X/S)$ is an infinite countable set that is not constructible.
		
		Let us remark that $f: (\X, -N_\X) \to S$ is a lc-trivial fibration with the trivial moduli part, and the generic fiber is klt. Here, $N_\X$ is the exceptional divisor for $\sigma$. Hence, the assumption in \cref{th: Ambro_isotriviality} on the effectiveness of the boundary can not be relaxed.
	\end{example}
	
	Second, we show that a set $\Iso_X(\X/S)$ does not have to be (Zariski) constructible in the analytic category even for minimal models, e.g. non-projective K3 surfaces.
	
	\begin{example}
		\label{ex: k3}
		Let $\Lambda_{\text{K3}} = E_8(-1)^{\oplus 2}\oplus U^{\oplus 3}$ be the K3 lattice. By \cite[Chapter 6, Section 3]{Huy16}, there exists a universal family $f: \X\to S$ (together with a marking $R^2f_*\underline{\Zi} \to \underline{\Lambda_{\textnormal{K3}}}$) of all smooth analytic marked K3 surfaces. The base $S$ is a 20-dimensional non-Hausdorff complex manifold. Let $\Phi: S\twoheadrightarrow \D^\text{K3}\subseteq \Pp({\Lambda_{\text{K3}}}\otimes_\Zi \Ko)$ be the period map, and $s_0$ be a (very general) point $S$ such that the orbit of $\Phi(s_0)$ under the natural action of $O(\Lambda_{\textnormal{K3}})$ is infinite. Set $X = \X_{s_0}$. By the Global Torelli Theorem for K3 surfaces, $s\in \Iso_X(\X/S)$ if and only if there is $\theta\in O(\Lambda_{\textnormal{K3}})$ such that $\Phi(s) = \theta \cdot \Phi(s_0)$. Hence, $\Iso_{X}(\X/S)$ is an infinite countable set that is not (Zariski) constructible. 
	\end{example}
	
	Finally, we provide an example of a projective lc $0$-pair $(X, B)$ showing that \cref{intro-theorem: orbits}, \cref{intro-theorem: finiteness}, \cref{intro-theorem: const} could not be generalized to the lc case. In fact, the construction is quite similar to \cref{ex: k3}.  This counterexample came out of a conversation with Philip Engel.
	\begin{example}
		\label{ex: anticanonical}
		We refer a reader to \cite{GHK15} for relevant definitions. Let $C\subset \Pp^2_\Ko$ be a nodal cubic. Then $(\Pp^2_\Ko, C)$ is a projective lc $0$-pair. Choose $9$ very general regular points $p_0, p_1, \dots, p_8$ on $C$, in particular, no three of $p_i$ are collinear, no six of $p_i$ on a conic. Then the rational surface $X=\text{Bl}_{p_0, p_1, \dots p_8} \Pp^2_\Ko$ includes no $(-2)$-curves. Let $B$ be the strict transform of $C$ to $X$. Then $(X, B)$ is a \textit{generic} Loijenga pair \cite[Definition 1.4]{GHK15}. Put $$S = \text{Hom}(\text{Pic}(X), \text{Pic}^0(B)), \ \text{Adm}_X = \{\theta\in O(\text{Pic}(X)): \theta(\text{Nef}(X)) = \text{Nef}(X), \ \theta(B) = B\}.$$
		
		Recall that $\text{Pic}(X)$ is torsion free, so $S\approx \mathbb{G}_m^{10}$. By \cite[Construction 5.7]{GHK15} there exists a projective elementary family $(\X/S, \B)$ together with a marking $\mu: \text{Pic}(X) \to \text{Pic}(\X)$ preserving the boundary classes. Moreover, $(\X/S, \B)$  is universal in the following sense: for every closed point $s\in S$, the marked period point of $(\X_s, \B_s)$ is $s$. In particular, $(X, B) = (\X_{s_0}, \B_{s_0})$ for some closed point $s_0\in S$. There is a map $\Phi: S\twoheadrightarrow \text{Hom}(B^\perp, \text{Pic}^0(B))$ onto the \textit{period domain of Loijenga pairs}, where $B^\perp = \{[D]\in \text{Pic}(X): [D]\cdot B = 0\}$. By the Global Torelli Theorem for Loijenga pairs \cite[Theorem 1.8]{GHK15}, we have that $s\in \Iso_{(X, B)}(\X/S, \B)$ if and only if there is $\theta \in\text{Adm}_X$ such that $\Phi(s) = \theta\cdot \Phi(s_0)$. We claim that the group $\text{Adm}_X$ is infinite. If this is the case, then the set $\Iso_{(X, B)}(\X/S, \B)$ is a countable union of proper subsets that is not constructible.
		
		Now, we prove the claim. Let $E\subset \Pp^2_\Ko$ be a smooth cubic intersecting $C$ in distinct points $p_0^\prime, p_1^\prime, \dots, p_8^\prime$ such that $p_i^\prime - p_0^\prime$ is non-torsion on $(E, p_0^\prime)$ for some $i = 1, 2, \dots, 8$. Then the rational surface $X^\prime = \text{Bl}_{p_0^\prime, p_1^\prime, \dots, p_8^\prime} \Pp^2_\Ko$ admits an elliptic fibration $f: X^\prime \to \Pp^1_\Ko$ such that $\text{rk}\, \text{MW}(X^\prime/\Pp^1_\Ko) \ge 1$. Note that the fibration $X^\prime/\Pp^1_\Ko$ is relatively minimal as $K_{X^\prime} = -B^\prime$, here $B^\prime $ is the strict transform of $C$ to $X^\prime$. Hence, every automorphism of the generic fiber of $f$ gives a non-trivial automorphism for $X^\prime$ over $\Pp^1_\Ko$. In particular, the group $\text{Aut}(X^\prime, B^\prime)$ is infinite. Now, we choose a marking $\text{Pic}(X^\prime) \to \text{Pic}(X)$. By construction of the universal family, there is a point $s^\prime \in S$ such that $(X^\prime, B^\prime) = (\X_{s^\prime}, \B_{s^\prime})$. This implies that $(X^\prime, B^\prime)$ is deformation equivalent to $(X, B)$. Let $\text{Hodge}_{X^\prime}\subseteq \text{Adm}_X$ denote the stabilizer of the period point $s^\prime$, then we have an exact sequence \cite[Theorem 5.1]{GHK15}:
		\begin{equation*}
			1 \to \text{Hom}(N^\prime, \mathbb{G}_m) \to \text{Aut}(X^\prime, B^\prime)\to \text{Hodge}_{X^\prime}/W_{X^\prime} \to 1,
		\end{equation*}  
		where $N^\prime = \text{coker} \left(\text{Pic}(X^\prime) \to \Zi\right)$, $[D]\mapsto [D]\cdot B^\prime$, and $W_{X^\prime}$ is the Weyl group. It is clear that $N^\prime = 0$. Then $\text{Hodge}_{X^\prime}/W_{X^\prime}$ is infinite. Hence, the group $\text{Adm}_X$ is infinite, as claimed.
		
		Thus, the natural analogue of \cref{intro-theorem: const} does not hold for $(X, B)$. According to \cref{th: constructibility_criterion}, the analogue of \cref{intro-theorem: orbits} does not hold for $(X, B)$ as well. Note that all fibers $(\X_s, \B_s)$ are $0$-pairs that $0$-equivalent to $(\Pp^2_\Ko, C)$. In addition, the fibers admit a bounded polarization. It follows that the analogue of \cref{intro-theorem: finiteness} can not hold for $(X, B)$.
	\end{example}
	\begin{remark}
		In addition, the failure of \cref{intro-theorem: orbits} (resp. \cref{intro-theorem: finiteness}) implies that the analogue of \cref{conj: KMT-Arithmetic} (resp. \cref{conj: KMT-Geometric}) does not hold. Nevertheless, the modified \cref{conj: KMT-Arithmetic} holds for $(X, B)$, that is, the natural action of $\text{Adm}_X$ on $\text{Nef}^e(X)$ admits a rational polyhedral fundamental domain \cite[Theorem 1.1]{Li25}.
	\end{remark}
		The example below is mentioned in \cite{Tot10} as an example of failure of the cone conjecture for lc $0$-pairs. We show that it also produces an example of failure of \cref{intro-theorem: orbits}, \cref{intro-theorem: finiteness}, \cref{intro-theorem: const} for plt $0$-pairs. We say a log pair $(X, B)$ is \textit{plt} if $\text{dis}(X, B) > -1 $.
	\begin{example}
		Let $C \subset \Pp^2_\Ko$ be a unique smooth cubic passing through 9 general points $p_i \in \Pp^2_\Ko$. Then $(\Pp^2_\Ko, C)$ is a plt $0$-pair. By \cite{Nag60}, the rational surface $X=\text{Bl}_{p_0, p_1, \dots p_8}$ includes infinitely many $(-1)$-curves. Furthermore, $\text{Aut}(X) = \{\text{id}\}$ due to \cite[Proposition 8]{Giz81}. Let $B$ be the strict transform of $C$ to $X$. Then $(X, B)$ is a plt $0$-pair. For each 3 distinct points $p_i, p_j, p_k$,there is a quadratic transformation $\varphi_{ijk}: \Pp^2_\Ko\dashrightarrow \Pp^2_\Ko$ centered at these points. We shall associate a linear transformation $\widetilde{\varphi_{ijk}}$ on $N^1_\R(X)$ to $\varphi_{ijk}$ as follows. Let $\sigma_{ijk}: \text{Bl}_{p_ip_jp_k} \Pp^2_\Ko\to \Pp^2_\Ko$ be the blow-up of $\Pp^2_\Ko$ at these points, and $E_i, E_j, E_k$ be the exceptional $(-1)$-curves. By $l_i\subset \text{Bl}_{p_ip_jp_k} \Pp^2_\Ko$ we denote the strict transform of the curve passing through $p_j, p_k$. The curves $l_j, l_k$ are defined similarly. Then the composition $\pi_{ijk} = \varphi_{ijk}\circ\sigma_{ijk}$ is a morphism contracting $(-1)$-curves $l_i, l_j, l_k$. Hence, in $\text{NS}(\text{Bl}_{p_ip_jp_k})$ there are two bases $(E_{ijk}, E_i, E_j, E_k)$ and $(l_{ijk}, l_{i}, l_j, l_k)$, where $E_{ijk}$ (resp. $l_{ijk}$) is the pullback of a general line in $\Pp^2_\Ko$ under $\sigma_{ijk}$ (resp. $\pi_{ijk}$). The change of basis is provided by the following matrix:
		\begin{equation*}
			\widetilde{\varphi_{ijk}} = \begin{pmatrix}
				2 & 1 & 1 & 1\\
				-1 & 0 & -1 & -1\\
				-1 & -1 & 0 & -1\\
				-1 & -1 & -1 & 0
			\end{pmatrix}.
		\end{equation*}
		 Since $\text{NS}(X) = \text{NS}(\text{Bl}_{p_ip_jp_k} \Pp^2_\Ko) \bigoplus_{n\neq i,j,k}\Zi[E_n]$, the linear transformation $\widetilde{\varphi_{ijk}}$ trivially extends to a linear transformation on $\text{NS}(X)$. Note that $\widetilde{\varphi_{ijk}}$ is the reflection with respect to the $(-2)$-vector $\alpha_{ijk} = E_{ijk}-E_i-E_j-E_k$. Hence, $\widetilde{\varphi_{ijk}}$ preserves the intersection form. Moreover, the class of $B = -K_X$ is invariant under $\widetilde{\varphi_{ijk}}$.
		
		Let $W\le \text{O}(1,9)\le \text{Gl}_{10}(\Zi)$ be the group generated by all $\varphi_{ijk}$. The group $W$ is known to be infinite \cite{Dol83}. We shall use the group $W$ to construct infinitely many polarized pairs $(X, B; A_i)$ such that $A_i$ is very ample, $A_i^2$ is bounded, and $A_i \not\equiv A_j$ for all $i\neq j$. According to \cite[Theorem 1.3]{Dur25} the cone $\text{Nef}(X)$ admits a rational polyhedral fundamental domain $\Pi\subset \text{Nef}(X)$ for the action $W$. Choose an ample divisor $A\in \Pi$. Let $I$ be an index set such that $\{A_i\}_{i\in I} = \{\varphi(A): \varphi\in W\}$. Then $A_i^2 = A^2$, and $B\cdot A_i = B\cdot A$ for all $i\in I$ as $W\le \text{O}(1,9)$. By Matsusaka's big theorem \cite{Mat72}, there is an integer $m\in \Zi_{>0}$ such that $A_i^{\otimes m}$ is very ample for all $i\in I$. According to \cref{prop: bounded_subclass}, there exists an projective elementary family $(\X/S, \B)$ together with very ample invertible sheaf $\A$ on $\X/S$ such that for infinitely many $i\in I$ there are closed points $s_i$: $(\X_{s_i}, \B_{s_i}, \A_{s_i})\approx (X, B, A_i^{\otimes m})$. Recall that $\Aut(X, B)$ is trivial. Thus, the set $\Iso_{(X, B)}(\X/S, \B)$ is not constructible in accordance with the criterion of constructibility (cf. \cref{cor: criterion_pairs}). As in \cref{ex: anticanonical}, \cref{intro-theorem: orbits}, \cref{intro-theorem: finiteness}, \cref{intro-theorem: const} could not be generalized to the case of plt $0$-pairs.
	\end{example}
	\begin{remark}
		In \cite{Tot10}, there is an example of a lc $0$-pair $(X, 0)$ for which the analogue of the cone conjecture does not hold. The example is based on \cite[Example 6.10]{DZ01}. One may prove that the analogues of  \cref{intro-theorem: orbits}, \cref{intro-theorem: finiteness}, \cref{intro-theorem: const} do not hold for $(X, 0)$ as well.
	\end{remark}
	We finish this paper by providing an example illustrating the behavior of isomorphic fibers over  fields that are not algebraically closed.
	\begin{example}
		Suppose $k = \R$. Let $n\ge 2$ be a positive even integer. Then $$X = \text{Proj}_k\,\frac{k[x,y,z]}{(x^n+y^n-z^n)}$$ is a smooth projective integral scheme over $k$, and $\dim X  = 1$. Moreover, $X\otimes_k \overline{k}$ is of general type (resp. Fano, Calabi--Yau) if $n\ge4$ (resp. $n=2, n=3$). In addition, $X(k)\neq \emptyset$. Set $$  \X  =\text{Proj}_k\, \frac{k[s,x,y,z]}{(x^n+y^n-sz^n)}, \ S =\text{Spec}_k\, k[s].$$ Let $f:\X\to S$ be the natural projection; $\text{Iso}_X(\X/S) = \{s\in S(k): \X_s \approx X\}$. Then $\text{Iso}_X(\X/S) = \{s>0\}\subset S(k)$ is not constructible. Nevertheless, the set $$\{s\in S(k): \X_s\otimes_k \overline{k} \approx X\otimes_k \overline{k}\} = \{s\in S(k): s\neq 0\}$$ is constructible (open). In fact, this follows from \cref{th: Fano_0_General}.
	\end{example}

	\newpage
	\appendix
	\section{}{\label{Appendix Cone_Conjecture}}
	In this Appendix, we shall explore an interconnection between the Kawamata--Matsuki conjecture on the finiteness of minimal models and the Kawamata--Morrison--Totaro cone conjecture. Throughout, the base field is $\Ko$. We say $X$ is a \textit{minimal model} if it is a proper $\Q$-factorial variety with terminal singularities such that the canonical class $K_X$ is nef. The following conjecture was proposed by Kawamata and Matsuki \cite[Conjecture 12.3.6]{Mat02}.
	
	\begin{conjecture}
		\label{conj: KM-finiteness_original}
		The number of isomorphism classes of projective minimal models in a fixed birational class is finite.
	\end{conjecture}
	
	This conjecture was generalized (see \cref{conj: KM-finiteness} below) to the case of log pairs in \cite{Ser26}, and proven for log surfaces (with non-zero boundary) independently of \cite{Tot10}. Recall that in \cite{Tot10} the cone conjecture is established in full generality in dimension $2$. Let us remark that \cref{conj: KM-finiteness} is proven for wlc models of general type \cite[Theorem E]{BCHM10}. Meanwhile, even \cref{conj: KM-finiteness_original} is an open problem in dimension $\ge 3$.
	
	\begin{conjecture}[Finiteness of minimal models]
		\label{conj: KM-finiteness}
		The number of (log) isomorphism classes of projective wlc klt models in a fixed $0$-class is finite.
	\end{conjecture}
	
	From now and on, we fix $(X, B)$ a projective wlc model (see \cref{def: 0-pair}) of dimension $d\in \Zi_{\ge 2}$. Let us remember that $B$ is effective.
	We shall assume that the \textit{abundance conjecture} holds for $(X, B)$, that is, there exists a normal variety $Y$, and a projective contraction $f: X\to Y$ such that $K_{X} + B \sim_{\R, Y} 0$. Denote $f$ by $X/Y$. We also assume that the \text{Abundance conjecture} holds for all projective wlc models of $(X, B)$.
	
	Let $D_1, D_2$ be Cartier divisors on $X$. Then we write $D_1\equiv_Y D_2$ if $D_1\cdot C = D_2\cdot C$ for every curve $C\subset X$ contracted by $f$. Let us define an $\R$-vector space $N^1_\R(X/Y)$: $$N^1_\Zi(X/Y) = \textnormal{NS}(X) / \equiv_Y, \ N^1_\Q(X/Y) = N^1_\Zi(X/Y)\otimes_\Zi \Q, \ N^1_\R(X) = N^1_\Q(X/Y) \otimes_\Q \R.$$ There are several natural convex cones in $N^1_\R(X/Y)$ that capture geometry of $X/Y$. Recall that a Cartier divisor on $X$ is called \textit{effective} over $Y$ (resp. \textit{nef} over $Y$, \textit{movable} over $Y$) if $f_*\Os(D)\neq 0$ (resp. if $D\cdot C \ge 0$ for every curve $C\subset X$ contracted by $f$, if $\text{supp}\left(\text{coker}\left(f^*f_*\Os_X(D)\to \Os_X(D)\right)\right)$ has codimension $\ge 2$). By $\textnormal{Eff}(X/Y)\subset N^1_\R(X/Y)$ denote the convex cone generated by all classes of effective Cartier divisors over $Y$. Let $\Amp(X/Y) \subset N^1_\R(X/Y)$ (resp. $\Mov(X/Y)\subset N^1_\R(X/Y)$) be the convex cone generated by all ample (resp. movable) classes over $Y$. Note that the closure $\Nef(X/Y):= \overline{\Amp(X/Y)}$ (in the Euclidean topology on $N^1_\R(X/Y)$) is the convex cone generated by all classes of nef over~$Y$ $\R$-Cartier $\R$-divisors due to Kleiman's criterion. Put
	\begin{equation*}
		\text{Nef}^e(X/Y) = \Nef(X/Y)\cap \textnormal{Eff}(X/Y), \ \overline{\Mov}^e(X/Y) = \overline{\Mov(X/Y)}\cap \textnormal{Eff}(X/Y).
	\end{equation*}
	By $\Mov^+(X/Y)$ we denote the convex hull of the set $\overline{\Mov(X/Y)}\cap N^1_\Q(X/Y) \subset N^1_\R(X/Y)$. The cone $\text{Nef}^+(X/Y)$ is defined similarly. Put
	\begin{equation*}
		\begin{gathered}
			\text{Aut}(X/Y, B) = \{\varphi\in \text{Aut}(X/Y, B): \varphi_*(B) = B\}, \\ \text{PsAut}(X/Y, B) = \{\varphi\colon X \overset{\textnormal{small}}{\longdashleftrightarrow} X \text{ over } Y|\, \varphi_*(B) = B\}.
		\end{gathered}
	\end{equation*}
	Now, we are ready to state the \textit{Kawamata--Morrison--Totaro cone conjecture}, which consists of two parts. We refer an interested reader to \cite{GLSW26} and the references therein for an overview, history, and recent progress.
	\begin{conjecture}[Arithmetic cone conjecture]
		\label{conj: KMT-Arithmetic}
		Suppose $(X, B)$ is klt. Then
		\begin{itemize}
			\item 
				There is a rational polyhedral cone $\Pi\subseteq \textnormal{Nef}^e(X/Y)$ such that
				\begin{equation*}
					\textnormal{Aut}(X/Y, B) \cdot \Pi = \textnormal{Nef}^e(X/Y) = \textnormal{Nef}^+(X/Y),
				\end{equation*}
				and for every $\varphi \in \textnormal{Aut}(X, B)$ we have  $\textnormal{Int}(\Pi)\cap \textnormal{Int}(\varphi^*\Pi) = \emptyset$ unless $\varphi^* = \textnormal{id}$.\\
			\item
				There is a rational polyhedral cone $\Pi^\prime\subseteq \overline{\textnormal{Mov}}^e(X/Y)$ such that
				\begin{equation*}
					\textnormal{PsAut}(X/Y, B) \cdot \Pi^\prime = \overline{\textnormal{Mov}}^e(X/Y) = \overline{\textnormal{Mov}}^+(X/Y),
				\end{equation*}
				and for every $\varphi \in \textnormal{PsAut}(X, B)$ we have  $\textnormal{Int}(\Pi^\prime )\cap \textnormal{Int}(\varphi^*\Pi^\prime) = \emptyset$ unless $\varphi^* = \textnormal{id}$.\\
		\end{itemize}
	\end{conjecture}
	\begin{remark}
		The rational polyhedral cone $\Pi$ (resp. $\Pi'$) in \cref{conj: KMT-Arithmetic} is also called \textit{rational polyhedral fundamental domain for} the natural (right) action of $\textnormal{Aut}(X/Y, B)$ on $\textnormal{Nef}^e(X/Y)$ (resp. $\textnormal{PsAut}(X/Y, B)$ on $\overline{\textnormal{Mov}}^e(X/Y)$).
	\end{remark}	
	Suppose $X$ is $\Q$-factorial. A small $\Q$-factorial modification (SQM) for $X/Y$ is a small birational map $\alpha: X \dashrightarrow X^\prime$ over $Y$ such that $X^\prime$ is a projective $\Q$-factorial variety over $Y$. Note that $\alpha_*: N^1_\Zi(X/Y) \to N^1_\Zi(X^\prime/Y)$ is an isomorphism. Moreover, we have $\alpha_* \Mov^e(X/Y) = \Mov^e(X^\prime/Y)$. Then we can define a chamber $\textnormal{A}^e(X^\prime/Y) = (\alpha^{-1})_* \Nef^e(X^\prime/Y) \subseteq \Mov^e(X/Y)$ corresponding to the SQM $X^\prime/Y$ for $X/Y$. If we assume the \textit{termination of flips} in relative dimension $\le d$, then we obtain the following:
	\begin{equation*}
		 \Mov^e(X/Y) = \bigcup_{\al: X\dashrightarrow X^\prime} \textnormal{A}^e(X^\prime/Y),
	\end{equation*}
	where $\al: X\dashrightarrow X^\prime$ runs over all distinct SQMs for $X/Y$. Furthermore, $\textnormal{A}^e(X^\prime/Y)\cap \textnormal{A}^e(X^{\prime\prime}/Y) \neq \emptyset$ if and only if $\textnormal{A}^e(X^\prime/Y) = \textnormal{A}^e(X^{\prime\prime}/Y)$, that is, $X^\prime$ is isomorphic to $X^{\prime\prime}$ over $Y$ (see \cite{Kaw97}).
	\begin{conjecture}[Geometric cone conjecture]
		\label{conj: KMT-Geometric}
		Suppose $(X, B)$ is klt, and $X$ is $\Q$-factorial. Then
		\begin{itemize}
			\item 
				The number of $\Aut(X/Y, B)$-equivalence classes of faces of the cone $\Nef^e(X/Y)$ corresponding to birational contractions or fiber space structures is finite.
			\item 
				The number of $\textnormal{PsAut}(X/Y, B)$-equivalence classes of chambers $\textnormal{A}^e(X^\prime/Y) \subseteq \Mov^e(X/Y)$ corresponding to the SQMs $\alpha: X\dashrightarrow X^\prime$ for $X/Y$ is finite.
		\end{itemize}
	\end{conjecture}
	\begin{proposition}
		Suppose $(X, B)$ is klt, $X$ is $\Q$-factorial. Suppose \cref{conj: KMT-Arithmetic} holds for $(X, B)$. Then \cref{conj: KMT-Geometric} holds for $(X, B)$.
	\end{proposition}
	\begin{proof}
		The first part of \cref{conj: KMT-Geometric} follows directly from \cref{conj: KMT-Arithmetic}. Therefore, we only need to prove that $(X, B)$ has only finitely many SQMs up to pseudo-automorphism. The argument is due to \cite[Theorem 2.14]{CL14}, which is based on Shokurov's log geography. Without further notice, we follow the notation in \cite{CS11}. By the hypothesis, there is a rational polyhedral fundamental domain $\Pi^\prime\subseteq \overline{\textnormal{Mov}}^e(X/Y)$ for the natural action $\text{PsAut}(X/Y, B)$. Let $D_1, D_2, \dots, D_r$ be effective divisors whose classes generate $\Pi^\prime$, let $S_1, S_2, \dots, S_k$ be all (distinct) prime divisors in the support of $D_1 + D_2 + \dots + D_r$. For $S = \sum_{i=1}^k S_i$ define the unit cube $\mathfrak{B}_S = \bigoplus_{i=1}^k [0,1] S_i$. By \cite[Theorem 3.4]{CS11} the set $\mathfrak{N}_S = \{D\in \mathfrak{B}_S: (X/Y, D) \text{ has a wlc model}\}$ is decomposed into a finite number $\sim_\text{wlc}$ classes $\mathfrak{P}$ (see \cite[Section 3]{CS11}). Furthermore, each class $\mathfrak{P}$ is a convex rational polyhedron. We shall consider only polyhedrons $\mathfrak{P}$ such that any $D\in \mathfrak{P}$ is big. Fix one such polyhedron~$\mathfrak{P}$. Let $C$ be a convex cone generated by $\overline{\mathfrak{P}}$ in the $\R$-vector space $\mathfrak{D}_S = \bigoplus_{i=1}^k\R S_i$. Note $C$ is a rational polyhedral cone, and $\mathcal{S} = C\cap \text{Div}(X)$ is a finitely generated monoid. It follows from $K_X + B\sim_{f, \R} 0$ that the ring $\mathfrak{R} = R(X, C) = \bigoplus_{D\in \mathcal{S}} H^0(X, D)$ is finitely generated. Then by \cite[Theorem 4.2]{KKL16} there are finitely many cones $C_j\subseteq C$ and contractions $f_j: X \dashrightarrow X_j$ such that $C = \bigcup_j \overline{C_j}$ and $f_j$ is an ample model for any $D\in C_j$. Moreover, $\overline{C_j}$ is rational polyhedral, and $f_j$ is a semiample model for any $D\in \overline{C_j}$. Therefore, for every $[D]\in \Pi^\prime\cap\text{Big}(X/Y) \subseteq \overline{\textnormal{Mov}}^e(X/Y)\cap\text{Big}(X/Y) $ there is an index $j$ such that $X_j \cong \text{Proj}_Y(X/Y, D^\prime)$ for some effective $D^\prime \equiv D$.
		This implies the proposition because of the equality $\text{PsAut}(X/Y, B)\cdot \Pi^\prime = \overline{\Mov}^e(X/Y)$, and  because of invariance for canonical rings under pseudo-automorphisms.
	\end{proof}
	
	\begin{proposition}
		\label{prop: geometric_implies_finiteness}
		Suppose $(X, B)$ is klt. Suppose \cref{conj: KMT-Geometric} holds for all projective $\Q$-factorial trm wlc models $0$-equivalent to $(X, B)$. Then \cref{conj: KM-finiteness} holds for $(X, B)$.
	\end{proposition}
	\begin{proof}
		Let $(X_\al, B_\al)$ be a projective wlc model of $(X, B)$. As we assumed that the Abundance conjecture holds for all projective wlc models of $(X, B)$, there exists a contraction $f_\al: X_\al \to Y$. Let $t_\al: (X^\prime_\al, B^\prime_\al) \to (X_\al, B_\al)$ be a $\Q$-factorial terminalization. To be definite, we may assume there is a single pair $(X^\prime, B^\prime)$ such that $(X^\prime_\al, B^\prime_\al) = (X^\prime, B^\prime)$ because \cref{conj: KMT-Geometric} holds for $(X, B)$, by the hypothesis. We can assume that all contractions $t_\al$ are not isomorphisms. Then the cone $t_\alpha^* \Nef^e(X_\al/Y)$ is a face of  $\Nef^e(X^\prime/Y)$. By the hypothesis, the number of $\Aut(X^\prime/Y, B^\prime)$-equivalence classes of faces of the cone $\Nef^e(X^\prime/Y)$ corresponding to birational contractions is finite. This concludes the proof.
	\end{proof}

	\begin{proposition}
		\label{prop: arithmetic}
		Suppose $(X, B)$ is klt.  Fix $N\in \Zi_{>0}$. Suppose \cref{conj: KMT-Arithmetic} holds for $(X, B)$. Then the number of orbits for the natural action of $\textnormal{Aut}(X, B)$ on the set $$\{[A]\in N^1_{\Zi}(X/Y) : A^{d} = N, \ A \text{ is ample over } Y\}$$ is finite (whenever the defined set is non-empty).
	\end{proposition}
	\begin{proof}
		The argument is based on \cite[Proposition 2.6]{Ste85}. By the hypothesis, there is a rational polyhedral fundamental domain  $\Pi\subseteq \textnormal{Nef}^e(X/Y)$ for $\text{Aut}(X, B)$. Let $x_1, x_2, \dots, x_d\in \Pi$ be linearly independent elements, then $x_{i_1}\cdot x_{i_2}\cdot \dots \cdot x_{i_d}\ge 0$ for every tuple $1\le i_1 \le i_2 \le \dots \le i_d\le d$, and $x_{i_1}\cdot x_{i_2}\cdot \dots \cdot x_{i_d}> 0$ if all indices are different. Since $\Pi\cap N^1_\Zi(X/Y)$ is a finitely generated semigroup, it follows that the set $\{[D]\in \Pi\cap N^1_\Zi(X/Y): D^d = N>0\}$ is finite. This implies the statement.
	\end{proof}
	
	\begin{remark}
		In accordance with \cref{th: finiteness of orbits(movable)}, the analogue of \cref{prop: arithmetic} should hold for the action of $\text{PsAut}(X/Y, B)$ on the set of big and movable over $Y$ Cartier divisors of fixed volume.
	\end{remark}

\end{document}